\documentclass[reqno,11pt]{amsart}
\usepackage{amsmath} 
\usepackage{amssymb} 
\usepackage{amsfonts} 
\usepackage{physics} 
\usepackage{esint} 
\usepackage{braket} 
\usepackage{enumitem}
\usepackage[toc,page]{appendix} 
\usepackage{hyperref}
\newtheorem{theorem}{Theorem}[section]
\newtheorem{proposition}[theorem]{Proposition}
\newtheorem{lemma}[theorem]{Lemma}
\newtheorem{corollary}[theorem]{Corollary}
\newtheorem{definition}[theorem]{Definition}
\theoremstyle{remark}
\newtheorem{remark}[theorem]{Remark}

\newcommand{\R}{\mathbb{R}}

\newcommand{\HH}{\mathcal{H}} 
\newcommand{\mres}{\mathbin{\vrule height 1.6ex depth 0pt width
0.13ex\vrule height 0.13ex depth 0pt width 1.3ex}}

\newcommand{\C}{\mathbb{C}}

\newcommand{\CC}[1]{\C #1:#1}
\newcommand{\NC}[1]{\abs{#1}_{\C}} 
\newcommand{\varepsilonA}{\varepsilon_{AR}} 

\author[C. Labourie]{Camille Labourie}
\author[A. Lemenant]{Antoine Lemenant}

\address[C. Labourie]{Université Paris-Saclay, CNRS, Laboratoire de mathématiques d'Orsay, 91405, Orsay, France.}
\email{camille.labourie@universite-paris-saclay.fr}

\address[A. Lemenant]{Institut universitaire de France and Universit\'e de Lorraine -- CNRS, UMR 7502 IECL, BP 70239
54506 Vandoeuvre-lès-Nancy,  France}
\email{antoine.lemenant@univ-lorraine.fr}


\date{\today}

\title{Uniform concentration property for Griffith almost-minimizers}

\begin{document}

\maketitle

\begin{abstract}
We prove that a Hausdorff limit of Griffith almost-minimizers remains a Griffith almost-minimizer. For this purpose, we introduce a new approach to the uniform concentration property of Dal Maso, Morel and Solimini which does not rely on the coarea formula, non available for symmetric gradient. We then develop several applications, including a general procedure to obtain global minimizers via blow-up limits.
\end{abstract}

\tableofcontents

\section{Introduction}

In recent years, a lot of attention has been given to the minimizers of the so-called Griffith functional, 
$$G(u,K):=\int_{\Omega \setminus K} \C e(u) : e(u) \dd x + \mathcal{H}^{N-1}(K),$$
defined on pairs function-set $(u,K)$, where  $K\subset \Omega\subset \R^N$ is a relatively closed set  and $u:\Omega\setminus K \to \R^N$ a displacement field. Here, $e(u)=(\nabla u +\nabla u^T)/2$ stands for the symmetrized gradient of $u$ and $\C$ is an elasticity tensor.

Since the functional is related to the variational model of crack propagation in linearized elasticity, it has been the central object of many works in the literature \cite{BCL,CC2,FLK2,BIL,CC,CFI,CFI2,CCF,CCI,cdens,dml,cfm1,CGP,dmft,dalmaso,LL,LL2}. Besides, the mathematical study of minimizers, that falls into the area of ``free-discontinuity problems'', brings a lot of technical difficulties compared to the well known scalar analogue, the Mumford-Shah functional.

A powerful approach to study the Griffith functional, which is usually referred to the ``weak formulation'', is to relax the problem  in the $GSBD$ space introduced by {\sc Dal Maso} in \cite{dalmaso}, where the pair $(u,K)$ is replaced by $u \in GSBD$ and $K = J_u$. Several existence and regularity results have been obtained in the $GSBD$ context in many recent papers (see for instance \cite{CCI, CFI, CC2}).

In this paper we shall not work in the $GSBD$ class but  work directly on pairs $(u,K)$. Our results apply for instance to the class of topological almost-minimizers for which $K$ may not be represented by the jump set of a $GSBD$ function. In this respect our work is more in the spirit of the approaches introduced for the Mumford-Shah functional by {\sc David}~\cite{DavidBOOK}, {\sc Bonnet}~\cite{Bonnet} or {\sc Dal Maso}, {\sc Morel} and {\sc Solimini}~\cite{DMS}.

The main contribution of the present paper is a limiting result for sequences of Griffith almost-minimizers converging with respect to the Hausdorff convergence of sets, see Theorem \ref{thm_limit}.
The difficulty in this context is to prove the semicontinuity behavior of the surface term.
This issue was already the main subject of previous works on the $GSBD$ space and was the key point in order to get the existence of a minimizer (see for instance \cite{dalmaso, CC}) but the literature does not deal with the convergence in the Hausdorff sense.
Yet, extracting converging sequences for the Hausdorff distance is instrumental for the regularity theory, as for instance to construct blow-up limits of minimizers, or to perform any argument by contradiction and compactness. We present at the end of the paper,  several applications.

We shall prove that Griffith almost-minimizers enjoy the so-called uniform concentration property, which was first introduced by {\sc Dal Maso}, {\sc Morel} and {\sc Solimini} \cite{DMS}, \cite{MSBOOK} in their work on the Mumford-Shah functional. 
This property says that every ball contains a smaller ball (but not too much smaller) where the density of the singular set is almost larger than $1$.
The point is to guarantee the lower-semicontinuity of the surface area along a converging sequence, similarly as in Golab's theorem.

The uniform concentration property for Mumford-Shah minimizers was established by {\sc Dal Maso}, {\sc Morel}, {\sc Solimini} \cite{DMS} in dimension $2$ and {\sc Maddalena}, {\sc Solimini} \cite{solimini}, \cite{MaSo1} in higher dimensions. Their technique, known as the excision method, does not extend to the symmetric gradient though because it relies on the full-gradient bound $\int_{B_r} \abs{\nabla u}^2 \dd{x} \leq C r^{N-1}$ to control the Hölder norm of $u$ outside of a thin neighborhood of the singular set.
Another approach is due to {\sc Rigot} \cite{Rigot}, who derived the uniform concentration from the uniform rectifiability of the singular set. However, the uniform rectifiability of Mumford-Shah minimizers \cite{DS96}, \cite{DS96bis} is proven via the co-area formula  which does not adapt to the symmetric gradient. It is worth mentioning that the piecewise Korn inequality \cite{Friedrich1}, \cite{Friedrich2} proved successful for substituting the co-area formula in the Griffith setting. This technique was utilized in \cite{FLK} to show almost-everywhere regularity and uniform rectifiability of Griffith minimizers, but it is limited to the dimension $2$.

We present a novel approach to the uniform concentration property which is suitable to the Griffith functional in any dimension.
It also yields a new proof of the uniform concentration in the scalar context of Mumford-Shah minimizers, which we believe is more elementary. It does not rely on the co-area formula, neither on powerful tools such as uniform rectifiability or the piecewise Korn inequality.
Here is one of our main result (we refer to Section \ref{definitions} for the Definition of a topological Griffith almost-minimizer).

\begin{theorem}[Uniform concentration property]\label{prop_uc}
    For each constant $\varepsilon \in (0,1)$, there exist constants $\varepsilon_0 > 0$ and $C_0 \geq 1$ (depending on $N$, $\C$, $\varepsilon$) such that the following holds.
    Let $(u,K)$ be a topological Griffith almost-minimizer with any gauge $h$ in $\Omega$. For all $x_0 \in K$ and for all $r_0 > 0$ such that $B(x_0,r_0) \subset \Omega$ and $h(r_0) \leq \varepsilon_0$, there exists $x \in B(x_0,r_0/2)$ and $r \in (C_0^{-1} r_0, r_0/2)$ such that
    \begin{equation*}
        \HH^{N-1}(K \cap B(x,r)) \geq (1 - \varepsilon) \omega_{N-1} r^{N-1},
    \end{equation*}
    where $\omega_{N-1}$ is the measure of the $(N-1)$-dimensional unit disk.
\end{theorem}


We now provide a brief overview of our proof of the uniform concentration property, in order to highlight the distinctive features of our work, for a specialist reader.
The principle is to use Carleson estimates to find many balls $B(x,r)$ where the elastic energy of $u$ is very small and to show that in such a ball,
\begin{equation}\label{eq_Kdensity}
    \HH^{N-1}(K \cap B(x,r)) \geq (1 - \varepsilon) \omega_{N-1} r^{N-1}.
\end{equation}
This latter point is given by Proposition \ref{prop_omega} and finds its intuition in the fact that the singular set of a Griffith minimizer behaves like a minimal sets, which are known to have density $\geq 1$, in regime of low elastic energy.

The proof of Proposition \ref{prop_omega} is by contradiction. After a suitable rescaling, we assume that there exists a sequence of almost-minimizers $(u_i,K_i)_i$ in $B(0,1)$ with vanishing elastic energy but with density uniformly bounded from above by $1 - \varepsilon$. We extract a subsequence which converges to a pair $(u,K)$ and we aim to show that the limit $K$ is a minimal set and that the area sequence is lower semi-continuous along the sequence.

For this purpose, our starting point is inspired by the works of {\sc Fang} \cite{Fang}, \cite{FangPHD} and a series of works by {\sc De Lellis et al.} \cite{I1}, \cite{I4} and {\sc De Philippis et al.} \cite{I2},  \cite{I5} on lower semi-continuity of the area for minimizing sequences of the Plateau problem. The key point of these works is to establish the rectifiability of the limit.
This is not straightforward as in general, a limit of rectifiable sets may not be rectifiable.
In the context of Mumford-Shah minimizers, the rectifiability along limits follows from the projection property introduced by {\sc Dibos}, {\sc Koepfler} \cite{dibos} and generalized in every dimension by {\sc Solimini} \cite{solimini}. As the proof is based on the excision method, it does not adapt to the symmetric gradient. Thus, one of the first difficulty of Proposition \ref{prop_omega} will be to prove that $K$ is rectifiable and this will done by use of a Federer-Fleming projection technique.

Thanks to  this approach, we are  reduced  to showing that (\ref{eq_Kdensity}) holds for a Griffith almost-minimizer in $B(x,r)$ in the situation where \emph{both} the flatness and the elastic energy are small.
In this case, the geometry of $K$ is under control via the flatness and this allows to bound the density by a constructive argument. This is done in Proposition \ref{prop_mainHole} and the proof consists in estimating the ``size of holes'' not directly for $K$, but for the orthogonal projection of $K$ onto a hyperplane.
Since the projection has less area, it is enough to bound from below the projection of $K$ in order to get \eqref{eq_Kdensity}. Then to estimate the projection, we first prove that under a small flatness and normalized energy, the normalized ``jump'' has to be greater than some threshold $\tau_0>0$. This is Lemma \ref{lem_Jinit} which is proved using the construction of a suitable competitor. Our notion of normalized jump in $B(x,r)$ is defined by
$$ J(x,r) := \frac{\abs{b_1 - b_2} + r|A_1 - A_2|}{\sqrt{r}},
$$
where $b_1+A_1x$ and $b_2+A_2x$ are two rigid movements that approximates $u$ above and below the approximative plane $P$ in $B(x,r)$. Then we can  estimate the size of holes for the projection of $K$ onto $P$, in $B(x,r)$, by integrating $u\cdot \nu$ along segments in the direction $\nu$ passing ``through'' the holes, where $\nu$ is orthogonal to $P$.  This is done in Lemma \ref{lem_slicing} and explains why we can avoid the use of the coarea formula. 

To be more precise, our argument has a degree of subtlety
because only one direction $\nu$ is admissible as passing ``through the holes'', and we cannot integrate along an almost vertical family of non colinear directions: therefore,  we choose one good  almost vertical direction that ``represents well'' the jump or in other words we slightly turn  the plane $P$ on which we project. By doing so, we lose a constant in the estimates, but since we have a universal control on the threshold $\tau_0>0$ which initializes the jump, the estimates are flexible enough to get the desired conclusion. This is done in Section \ref{section_mainHole} in the proof of Proposition \ref{prop_mainHole}. 

With the uniform concentration at hand, we can prove a general principle for limits of sequences of almost minimizers, as stated in a second main result Theorem~\ref{thm_limit}. We then use it to get several applications. The first one is a general strategy to take blow-up limits, and prove that any blow-up sequence must converge to a global minimizer, which is the purpose of Section~\ref{section_blowup}. In Proposition~\ref{prop_classification}, we prove that any global minimizer in dimension 2 whose singular set is a cone, must be a line, a half-line, or a triple junction. Finally, in Proposition~\ref{prop_dimension}, we extend the theorem of {\sc Ambrosio}, {\sc Fusco} and {\sc Hutchinson} \cite{AFH} to the Griffith setting. This result estimates the Hausdorff dimension of the singular set via the integrability exponent of the symmetric gradient.

Let us now introduce some definitions and state our main result more precisely.

\section{Definitions and statement of the main result}\label{definitions}
Our working space is an open set $\Omega \subset \R^N$, where $N \geq 2$.
We say that a constant is \emph{universal} when it depends only on $N$.
Given a set $A$, the notation $A \subset \subset \Omega$ stands for $\overline{A} \subset \Omega$.
We define a rigid motion as an affine map $a : \R^N \to \R^N$ of the form $a(x) = b + Ax$, where $b \in \R^N$ and $A \in \R^{N \times N}$ is a skew-symmetric matrix.

\medskip

\noindent{\bf Elasticity tensor.}
Given two matrices $\xi, \eta \in \R^{N \times N}$, the notation $\xi : \eta$ denotes the Frobenius inner product of $\xi$ and $\eta$,
\begin{equation*}
    \xi : \eta := \sum_{ij} \xi_{ij} \eta_{ij}.
\end{equation*}
The Frobenius norm is then given by $\abs{\xi} := \sqrt{\sum_{ij} (\xi_{ij})^2}$.
We fix for the whole paper a symmetric linear map $\C : \R^{N \times N} \to \R^{N \times N}$ such that for all $\xi \in \R^{N \times N}$,
\begin{equation*}
    \C(\xi - \xi^T) = 0 \quad \text{and} \quad \C \xi : \xi \geq c_0^{-1} \abs{\xi + \xi^T}^2,
\end{equation*}
for some constant $c_0 \geq 1$.
Note that $\C$ defines a scalar product on the space $\R^{N \times N}_{\mathrm{sym}}$ of symmetric matrices.

\medskip

\noindent{\bf (Coral) pairs.}
We define an \emph{admissible pair} as a pair $(u,K)$ such that $K$ is a relatively closed subset of $\Omega$ and $u \in W^{1,2}_{\mathrm{loc}}(\Omega \setminus K;\R^N)$.
We say that a pair has a \emph{locally finite energy} provided that for all ball $B \subset \subset \Omega$, 
\begin{equation*}
    \int_{B \setminus K} \abs{e(u)}^2 \dd{x} + \HH^{N-1}(K \cap B) < +\infty.
\end{equation*}
We say that a relatively closed set $K \subset \Omega$ is \emph{coral} if for all $x \in K$, for all $r > 0$,
\begin{equation*}
    \HH^{N-1}(K \cap B(x,r)) > 0,
\end{equation*}
where $\HH^{N-1}$ is the Hausdorff measure of dimension $N-1$.
We also say that a pair $(u,K)$ is coral when $K$ is coral.

\medskip

\noindent{\bf Competitors.}
Let $(u,K)$ be an admissible pair.
Let $B$ be an open ball such that $B \subset \subset \Omega$.
A \emph{competitor} of $(u,K)$ in $B$ is an admissible pair $(v,L)$ such that
\begin{equation}\label{eq_competitor}
    L \setminus B = K \setminus B \quad \text{and} \quad v = u \quad \text{a.e. in} \quad \Omega \setminus \left(K \cup B\right).
\end{equation}
Given a relatively closed set $K \subset \Omega$, a \emph{topological competitor}\footnote{We follow the terminology in \cite{DavidBOOK} for \emph{topological competitors}. Theses were first introduced by Bonnet \cite{Bonnet} and are also called  \emph{MS-competitors} in \cite{David2009} or \emph{separation competitors} in \cite{FLK2}.} of $K$ in $B$ is a relatively closed subset $L \subset \Omega$ such that $L \setminus B = K \setminus B$ and
\begin{equation}\label{eq_topo_competitor}
    \text{all points $x,y \in \Omega \setminus (K \cup B)$ which are separated by $K$ are also separated by $L$}.
\end{equation}
This means that if $x,y \in \Omega \setminus (K \cup B)$ belongs to different connected component of $\Omega \setminus K$, they also belong to different connected components of $\Omega \setminus L$.
We say that a pair $(v,L)$ is a topological competitor of $(u,K)$ if it is a competitor of $(u,K)$ as in (\ref{eq_competitor}) and if in addition, $L$ is a topological competitor of $K$ as in (\ref{eq_topo_competitor}).

\begin{remark}\label{rmk_borsuk}
    An example of topological competitors are sets of the form $L = f(K)$, where $f : K \to \R^N$ is a continuous map such that $f = \mathrm{id}$ in $K \setminus B$ and $f(K \cap B) \subset B$.
    \footnote{More precisely, the theory of Borsuk maps (\cite[Chap. XVII, 4.3]{dugundji}) states that if $A$ is a compact set of $\R^N$ which separates two points $p, q \in \R^N \setminus A$ and if $\phi : A \times [0,1] \to \R^N$ is a continuous map such that
    \begin{equation*}
        \phi(\cdot,0) = \mathrm{id} \quad \text{and} \quad p, q \notin \phi(A \times [0,1]),
    \end{equation*}
    then $p$ and $q$ are still separated by $\phi(A,1)$.
    Let us deduce that if two points $p,q \in \Omega \setminus \left(K \cup B\right)$ are separated by $K$, they are also separated by $f(K)$. We proceed by contradiction and assume that there exists a continuous path $\gamma$ connecting $p,q$ in $\Omega \setminus f(K)$. We let $V \subset \subset \Omega$ be an open set such that $\gamma \cup B \subset V$ and we consider the compact set
        $$A := (K \cap \overline{V}) \cup \partial V.$$
    Since $p,q$ belong to $V$ and lie in distinct connected components of $\Omega \setminus K$, they also lie in distinct connected components of $\R^N \setminus A$.
    We extend $f$ continuously on $A$ by setting $f = \mathrm{id}$ on $\partial V$. The function $f$ satisfies $f(A \cap B) \subset B$ and $f = \mathrm{id}$ in $A \setminus B$ so $p,q$ stay outside $\phi(A \times [0,1])$, where $\phi(x,t) = (1-t) x + tf(x)$.
    In particular, $p$ and $q$ lie in distinct connected components of $\R^N \setminus f(A)$ but this contradicts the fact that $\gamma$ is disjoint from $f(A) \subset f(K) \cup \partial V$.}
\end{remark}

\medskip

\noindent{\bf Quasiminimizers.}
We define a \emph{gauge} as a non-decreasing function $h : (0,+\infty) \to [0,+\infty]$ such that $h(r) <+\infty$ for sufficiently small $r$.
\begin{definition}[Quasiminimizers]\label{defi_quasi}
    Let $M \geq 1$ and let $h$ be a gauge. A \emph{Griffith local $M$-quasiminimizer} with gauge $h$ in $\Omega$ is a coral pair $(u,K)$ with locally finite energy such that for all $x \in \Omega$, for all $r > 0$ with $\overline{B}(x,r) \subset \Omega$ and for all competitor $(v,L)$ of $(u,K)$ in $B(x,r)$, we have
    \begin{multline*}
        \int_{B(x,r) \setminus K} \CC{e(u)} \dd{x} + M^{-1} \HH^{N-1}(K \cap B(x,r)) \\\leq \int_{B(x,r) \setminus L} \CC{e(v)} \dd{x} + M \HH^{N-1}(L \cap B(x,r)) + h(r) r^{N-1}.
    \end{multline*}
    Moreover,
    \begin{enumerate}[label=(\roman*)]
        \item a \emph{Griffith local minimizer} is a pair which satisfies the above definition with $M = 1$ and $h = 0$;
        \item a \emph{Griffith local almost-minimizer} is a pair which satisfies the above definition with $M = 1$ and a gauge $h$ such that $\lim_{r \to 0} h(r) = 0$;
        \item a \emph{Griffith local topological $M$-quasiminimizer} (resp. almost-minimizer or minimizer) is a pair which satisfies the above definition but only with respect to topological competitors.
    \end{enumerate}
\end{definition}

In the following, we omit the word ``local'' and ``Griffith'' for convenience.
Our terminology follows the spirit of \cite{DavidBOOK}.
Almost-minimizers look like a minimizer at small scales. They are expected to have fine regularity properties and one may hope to classify their local behaviors.
On the other hand, quasiminimizers form a much broader class which has bilipschitz invariant properties.
We stress that the gauge of an almost-minimizer satisfies $\lim_{r \to 0} h(r) = 0$ by definition whereas the gauge of a quasiminimizer is allowed not to go to zero when $r \to 0$. For example, the gauge of a quasiminimizer might be a small constant.
Our notion of quasiminimizer is larger than \cite[Definition 7.21]{DavidBOOK} so as to include the minimizers of a larger class of functionals, see
\cite[Theorem 2.7]{FLK2}. 
The notion of quasiminimizer in the book of {\sc Ambrosio}, {\sc Fusco}, {\sc Pallara} \cite{AFH} corresponds in our paper to an almost-minimizer with gauge $h(r) = h(1) r^{\alpha}$.

\begin{remark}[Standard rescaling of quasiminimizers]\label{rmk_scaling}
    If $(u,K)$ is a (resp. topological) $M$-quasiminimizer with gauge $h$ in a ball $B(x_0,r_0)$, then the pair $(v,L)$ in $B(0,1)$, defined by
    \begin{equation*}
        v(x) := r_0^{-1/2} u(x_0 + r_0 x) \quad \text{and} \quad L := r_0^{-1}(K - x_0),
    \end{equation*}
    is a (resp. topological) $M$-quasiminimizer with gauge $\tilde{h}(t) = h(r_0 t)$ in $B(0,1)$.
\end{remark}

\begin{definition}[Almost-minimal sets]\label{defi_minimalset}
    Let $h$ be a gauge such that $\lim_{r \to 0} h(r) = 0$.
    An almost-minimal set with gauge $h$ in $\Omega$ is a relatively closed and coral subset $K \subset \Omega$ such that for all $x \in K$, for all $r > 0$ such that $\overline{B}(x,r) \subset \Omega$ and for all topological competitor $L$ of $K$ in $B(x,r)$, we have
    \begin{equation*}
        \HH^{N-1}(K \cap B) \leq \HH^{N-1}(L \cap B) + h(r) r^{N-1}.
    \end{equation*}
    In the case $h = 0$, we say that it is a minimal set.
\end{definition}
This property says that a topological competitor $L$ for $K$ can decrease the area, but only up to a controlled error term.
There are also different notions of minimal sets in the literature such as Almgren minimal sets \cite{Almgren76} which are minimal under Lipschitz deformations.

\medskip

\noindent{\bf Ahlfors-regularity.} For each $M \geq 1$, there exist constants $\varepsilonA \in (0,1)$ and $C \geq 1$ (depending on $N$, $\C$, $M$) such that the following holds.
Let $(u,K)$ be a topological quasiminimizer with any gauge $h$ in $\Omega$.
Then for all $x \in \Omega \cap K$, for all $r > 0$ such that $B(x,r) \subset \Omega$ and $h(r) \leq \varepsilonA$, we have
\begin{equation}\label{eq_AF}
    \HH^{N-1}(K \cap B(x,r)) \geq C^{-1} r^{N-1}.
\end{equation}
For details, we refer to \cite{FLK2} which extends the method of \cite{CCI}, \cite{CFI} to topological quasiminimizers.
Up to choose $C$ a bit larger (still depending only on $N$, $\C$, $M$), it is standard that we also have that for all $x \in \Omega$ and $r > 0$ such that $B(x,r) \subset \Omega$,
\begin{equation}\label{eq_AF2}
    \int_{B(x,r)} \abs{e(u)}^2 \dd{x} + \HH^{N-1}(K \cap B(x,r)) \leq C(1 + h(r)) r^{N-1}.
\end{equation}
When $h(r) \leq \varepsilonA$, we directly assume that the right-hand side of (\ref{eq_AF2}) is bounded by $C r^{N-1}$.
A reasonable gauge should satisfy at least $\lim_{r \to 0} h(r) < \varepsilonA$ so that a quasiminimizer with gauge $h$ is locally Ahlfors-regular.
We will frequently refer to $\varepsilonA$ in the paper as we will need to assume that gauges are less than $\varepsilonA$ to take advantage of (\ref{eq_AF}), (\ref{eq_AF2}).

\medskip

\noindent{\bf Flatness.}
Let $(u,K)$ be a pair in $\Omega$.
For any $x_0 \in K$ and $r_0 > 0$ such that $B(x_0,r_0) \subset \Omega$, we define the flatness $\beta_K(x_0,r_0)$ of $K$ in $B(x_0,r_0)$ as
\begin{equation*}
    \beta_K(x_0,r_0) := \inf_{P} \sup_{x \in K \cap B(x_0,r_0)} \mathrm{dist}(x,P),
\end{equation*}
where $P$ runs among affine hyperplanes passing through $x_0$.
This is equivalently the infimum of all $\varepsilon > 0$ for which there exists an hyperplane $P$ through $x_0$ such that
\begin{equation*}
    K \cap B(x_0,r_0) \subset \set{y \in B(x_0,r_0) | \mathrm{dist}(y,P) \leq \varepsilon r_0}.
\end{equation*}
There always exists an hyperplane $P$ which achieves the infimum.
When there is no ambiguity, we write $\beta$ instead of $\beta_K$.
We can define similarly the bilateral flatness as
\begin{equation}\label{eq_bilateral_flatness}
    \beta^{\rm bil}_K(x_0,r_0) := \inf_{P} \max\left(\sup_{x \in K \cap B(x_0,r_0)} \mathrm{dist}(x,P),\sup_{x \in P \cap B(x_0,r_0)} \mathrm{dist}(x,K)\right).
\end{equation}
The flatness and the bilateral flatness are invariant under rescaling, see Remark \ref{rmk_scaling}

\medskip

\noindent{\bf Normalized elastic energy.}
For any $x_0 \in \Omega$ and $r_0 > 0$ such that $B(x_0,r_0) \subset \Omega$, we define the \emph{normalized elastic energy} of $u$ in $B(x_0,r_0)$ as
\begin{equation*}
    \omega(x_0,r_0) := r_0^{1 - N} \int_{B(x_0,r_0) \setminus K} \abs{e(u)}^2 \dd{x}.
\end{equation*}
More generally, for $p \geq 1$, we define
\begin{equation*}
    \omega_p(x_0,r_0):= r_0^{1-2N/p}\left(\int_{B(x_0,r_0)\setminus K} \abs{e(u)}^p \dd{x}\right)^{\frac{2}{p}}.
\end{equation*}
Here the exponent on the radius is chosen in such a way that $\omega_p$ is invariant under rescaling, see Remark \ref{rmk_scaling}.
Note that $\omega_2 = \omega$ and that for $p \in [1,2]$, we have $\omega_p \leq \omega$ by Hölder inequality.

\medskip

\noindent{\bf Local Hausdorff convergence of sets.}
We consider a sequence of open sets $(\Omega_i)_i \subset \R^N$ and an open set $\Omega$ such that
\begin{equation}\label{eq_Omega}
    \text{for all compact set $H \subset \Omega$, we have $H \subset \Omega_i$ for $i$ large enough.}
\end{equation}
\begin{definition}
    Let $(K_i)_i$ be a sequence such that for all $i$, $K_i$ is a relatively closed subset of $\Omega$.
    We say that $(K_i)_i$ \emph{converges in local Hausdorff distance} to a relatively closed subset $K \subset \Omega$ if for all compact set $H \subset \Omega$,
    \begin{equation*}
        \lim_{i \to +\infty} \left(\sup_{x \in K_i \cap H} \mathrm{dist}(x,K) + \sup_{x \in K \cap H} \mathrm{dist}(x,K_i)\right) = 0.
    \end{equation*}
\end{definition}
\noindent
This means for all $\varepsilon > 0$, there exists an index $i_0$ such that for all $i \geq i_0$,
\begin{equation*}
    K_i \cap H \subset \set{\mathrm{dist}(\cdot,K) \leq \varepsilon} \quad \text{and} \quad K \cap H \subset \set{\mathrm{dist}(\cdot,K_i) \leq \varepsilon}.
\end{equation*}
One can check that this convergence is equivalent to the two inclusions
\begin{align*}
    \set{x \in \Omega | \liminf_{i \to +\infty} \mathrm{dist}(x,K_i) = 0} \subset K \subset \set{x \in \Omega | \lim_{i \to +\infty} \mathrm{dist}(x,K_i) = 0}.
\end{align*}
Since the right-hand side is always a subset of the left-hand side, these inclusions are actually equalities and we have
\begin{equation*}
    K = \set{x \in \Omega | \lim_{i \to +\infty} \mathrm{dist}(x,K_i) = 0}.
\end{equation*}
As a consequence of the definition, we see that
\begin{equation}\label{eq_KCV}
    \text{for all compact set $H \subset \Omega \setminus K$, we have $H \subset \Omega_i \setminus K_i$ for $i$ big enough.}
\end{equation}
It follows from (\ref{eq_KCV}) that
\begin{equation}\label{eq_KCV2}
    \begin{gathered}
        \text{for all compact set $H \subset \Omega \setminus K$ and for all open set $V \subset \Omega$,}\\
        \text{if $K \cap H \subset V$, then we have $K_i \cap H \subset \Omega_i \cap V$ for $i$ big enough.}
    \end{gathered}
\end{equation}

\medskip

\noindent{\bf Convergence of pairs.}
We consider a sequence of open sets $(\Omega_i)_i \subset \R^N$ and an open set $\Omega$ such that
\begin{equation*}
    \text{for all compact set $H \subset \Omega$, we have $H \subset \Omega_i$ for $i$ large enough.}
\end{equation*}
\begin{definition}
    Let $(u_i,K_i)_i$ be a sequence such that for all $i$, $(u_i,K_i)$ is a pair in $\Omega_i$. We say that $(u_i,K_i)_i$ converges to a pair $(u,K)$ in $\Omega$ if
    \begin{enumerate}[label=(\roman*)]
        \item $(K_i)_i$ converges to $K$ in local Hausdorff distance;
        \item for all connected component $\mathcal{O}$ of $\Omega \setminus K$, there exists a sequence of rigid motions $(a_i)_i$ such that for all compact set $H \subset \mathcal{O}$,
            \begin{equation*}
                \lim_{i \to +\infty} \int_H \abs{u_i - a_i - u}^2 \dd{x} = 0.
            \end{equation*}
    \end{enumerate}
\end{definition}
This is the vectorial analogue of the convergence considered by {\sc Bonnet} \cite{Bonnet}.
Note that the limit displacement $u$ is only determined up to a rigid motion in each connected component of $\Omega \setminus K$.

\medskip

Now, here is the  main result of our paper.

\begin{theorem}\label{thm_limit}
    Let $(\Omega_i)_i$ and $\Omega$ be a sequence of open sets as in (\ref{eq_Omega}).
    Let $(u_i,K_i)_i$ be a sequence such that for all $i$, $(u_i,K_i)_i$ is a topological almost-minimizer with gauge $h_i$ in $\Omega_i$.
    We assume that $(u_i,K_i)_i$ converges to a pair $(u,K)$ in $\Omega$.
    We define for $r > 0$,
    \begin{equation*}
        h(r) =
        \begin{cases}
            \lim_{t \to r^+} \left(\limsup_{i} h_i(t)\right) &\text{if this quantity is $< \varepsilonA$}\\
            +\infty                                           &\text{otherwise.}
        \end{cases},
    \end{equation*}
    and we assume that $\lim_{t \to 0} \limsup_{i} h_i(t) = 0$.
    Then $(u,K)$ is a topological almost-minimizer with gauge $h$ in $\Omega$.
    Moreover, for all $x \in \Omega$ and $r > 0$ such that $\overline{B}(x,r) \subset \Omega$, we have
    \begin{align*}
        \liminf_{i \to +\infty} \int_{B(x,r) \setminus K_i} \CC{e(u_i)} \dd{x} &\geq \int_{B(x,r) \setminus K} \CC{e(u)} \dd{x}\\
        \limsup_{i \to +\infty} \int_{B(x,r) \setminus K_i} \CC{e(u_i)} \dd{x} &\leq \int_{B(x,r) \setminus K} \CC{e(u)} \dd{x} + h(r) r^{N-1}
    \end{align*}
    and
    \begin{align*}
        \liminf_{i \to +\infty} \HH^{N-1}(K_i \cap B(x,r)) &\geq \HH^{N-1}(K \cap B(x,r))\\
        \limsup_{i \to +\infty} \HH^{N-1}(K_i \cap \overline{B}(x,r)) &\leq \HH^{N-1}(K \cap \overline{B}(x,r)) + h(r) r^{N-1}.
    \end{align*}
    If furthermore
    \begin{equation*}
        \lim_{i \to + \infty} \int_H \abs{e(u_i)} \dd{x} = 0 \quad \text{for all compact set $H \subset \Omega \setminus K$},
    \end{equation*}
    then $u$ is a rigid motion in each connected component of $\Omega \setminus K$ and $K$ is an almost-minimal set with gauge $h$ in $\Omega$.
\end{theorem}

This result generalizes to the Griffith setting the known limiting theorems of the scalar case. The first theorem of this kind was due to {\sc Bonnet} \cite[Theorem 2.2]{Bonnet} for blow-up limits of Mumford-Shah minimizers in $\R^2$.
It was generalized to Mumford-Shah almost-minimizers in any dimension by {\sc Maddalena}, {\sc Solimini} \cite[Theorem 11.1]{MaSo4} and {\sc David} \cite[Theorem 38.3]{DavidBOOK}. The particular case where the Dirichlet energy goes to zero was also dealt with independently by {\sc Ambrosio}, {\sc Fusco}, {\sc Hutchinson} \cite[Theorem 5.4]{AFH} and {\sc De Lellis}, {\sc Focardi} \cite[Theorem 13]{DLF1}.

The assumption $\lim_{t \to 0} \limsup_{i} h_i(t) = 0$ makes sure that the limit gauge $h$ satisfies $\lim_{r \to 0} h(r) = 0$, as requested in the definition of almost-minimizers. The minimality properties of the limit are restricted to balls $B(x,r)$ such that $\limsup_{i} h_i(r) < \varepsilonA$ because this guarantees that the sequence $(u_i,K_i)_i$ is uniformly Ahlfors-regular in $B(x,r)$, see (\ref{eq_AF}) and (\ref{eq_AF2}).

Even if all the pairs $(u_i,K_i)_i$ are plain almost-minimizer (without the topological constraint on competitors), it is unavoidable that the limit may only be minimal with respect to topological competitors.
As an example, if one takes a blow-up limit of a Mumford-Shah minimizer $(u,K)$ at a smooth point $x_0 \in K$, the limit is a pair $(u_{\infty},K_{\infty})$ such that $K_{\infty}$ is an hyperplane and $u_{\infty}$ is piecewise constant.
It is known in this case that $(u_{\infty},K_{\infty})$ is a topological minimizer but not a plain minimizer as one can find a better competitor by making a hole with suitable dimensions (see the comment just before \cite[Proposition 6.8]{afp}).

\section{Preliminaries on limits}

\subsection{Standard properties}

We start this section by observing that the convergence of pairs is preserved under rescaling. We leave the details to the reader.

\begin{remark}\label{rmk_limit_scaling}
    Let $(\Omega_i)_i$ and $\Omega$ be a sequence of open sets as in (\ref{eq_Omega}).
    Let $(u_i,K_i)_i$ be a sequence such that for all $i$, $(u_i,K_i)$ is a pair in $\Omega_i$.
    We assume that $(u_i,K_i)_i$ converges to a pair $(u,K)$ in $\Omega$. Let us fix $x_0 \in \R^N$ and $r_0 > 0$. Then the sequence of pairs $(v_i,L_i)_i$ in $r_0^{-1}(\Omega_i - x_0)$ defined by
    \begin{equation*}
        v_i(x) = r_0^{-1/2} u_i(x_0 + r_0 x) \quad \text{and} \quad L_i = r_0^{-1}(K_i - x_0)
    \end{equation*}
    converge to $(v,L)$ in $r_0^{-1}(\Omega - x_0)$, where
    \begin{equation*}
        v(x) = r_0^{-1/2} u(x_0 + r_0 x), \quad \text{and} \quad L := r_0^{-1}(K - x_0).
    \end{equation*}
\end{remark}

We recall a standard compactness principle for the local Hausdorff convergence. This is a minor adaptation of \cite[Proposition 34.6]{DavidBOOK} and we omit the proof.
\begin{lemma}\label{lem_hausdorff_compactness}
    Let $(\Omega_i)_i$ and $\Omega$ be a sequence of open sets as in (\ref{eq_Omega}).
    Let $(K_i)_i$ be a sequence such that for all $i$, $K_i$ is a relatively closed subset of $\Omega$.
    Then there exists a subsequence which converges to a relatively closed subset $K$ of $\Omega$.
\end{lemma}


Then, we deduce a compactness principle for pairs.
\begin{lemma}\label{lem_compactness}
    Let $(\Omega_i)_i$ and $\Omega$ be a sequence of open sets as in (\ref{eq_Omega}).
    Let $(u_i,K_i)_i$ be a sequence such that for all $i$, $(u_i,K_i)$ is a pair in $\Omega_i$ and assume that for all $x \in \Omega$, there exists $r > 0$ such that $B(x,r) \subset \Omega$ and
    \begin{equation*}
        \limsup_{i \to +\infty} \int_{B(x,r) \setminus K_i} \abs{e(u_i)}^2 < +\infty.
    \end{equation*}
    Then there exists a subsequence of $(u_i,K_i)_i$ which converges to a pair $(u,K)$ in $\Omega$.
\end{lemma}
\begin{proof}
    By Lemma \ref{lem_hausdorff_compactness}, we can first extract a subsequence such that $(K_i)_i$ converges to a relatively closed subset $K \subset \Omega$.
    Now, we turn our attention to the functions $(u_i)_i$. 
    We fix a connected component $\mathcal{O}$ of $\Omega \setminus K$.
    We cover $\mathcal{O}$ by non-empty open balls $(B_n)_{n \geq 0}$ such that $B_n \subset \subset O$ and for all $n \geq 0$,
    \begin{equation*}
        \limsup_{i \to +\infty} \int_{B_n} \abs{e(u_i)}^2 \dd{x} < +\infty.
    \end{equation*}
    For all $n$, we have $B_n \subset \Omega_i$ for $i$ big enough and we observe using the Korn-Poincaré inequality that there exists a rigid motion $a_{n,i}$ such that
    \begin{equation*}
        \int_{B_n} \abs{u_i - a_{n,i}}^2 \dd{x} \leq C \mathrm{diam}(B)^2 \int_{B_n} \abs{e(u_i)}^2 \dd{x},
    \end{equation*}
    and
    \begin{equation*}
        \int_{B_n} \abs{\nabla u_i - \nabla a_{n,i}}^2 \dd{x} \leq C \int_{B_n} \abs{e(u_i)}^2 \dd{x}.
    \end{equation*}
    Therefore, we see that for all $n$, the sequence $(u_i - a_{n,i})_i$ is bounded in $W^{1,2}(B_n;\R^N)$ and by a diagonal extraction argument, we can extract a subsequence of $(u_i)_i$ (not relabelled) such that for all $n$, the sequence $(u_i - a_{n,i})_i$ converges in $L^2(B_n,\R^N)$ to a function in $W^{1,2}(B_n;\R^N)$.
    Now we let $a_i := a_{0,i}$ (the rigid motion in the ball $B_0$) and we are going to show that for all $n \geq 0$, the sequence $(a_n - a_{n,i})_i$ converges locally uniformly in $\R^N$ to a rigid motion.
    First we observe that for every $x \in \mathcal{O}$, there exists a finite chain of balls among $(B_n)_n$ linking $B_0$ to $x$. More precisely, there exists a finite number of indices $n(1),\ldots,n(l)$ with $n(1) = 0$ and $x \in B_{n(l)}$, such that for all $0 \leq k < l$, $B_{n(k)} \cap B_{n(k+1)} \ne \emptyset$. This is a consequence of connectedness as the set of points $x \in \mathcal{O}$ satisfying this property is non-empty and is both relatively open and closed in $\mathcal{O}$.
    Now, we fix a ball $B_n$ and by the above observation we can consider a finite number of indices $n(1),\ldots,n(l)$ with $n(1) = 0$ and $n(l) = n$, such that for all $0 \leq k < l$, $B_{n(k)} \cap B_{n(k+1)} \ne \emptyset$.
    Since $(u_i - a_{n(k),i})_i$ converges in $L^2(B_{n(k)};\R^N)$ and $(u_i - a_{n(k+1),i})_i$ converges in $L^2(B_{n(k+1)};\R^N)$, we deduce that $(a_{n(k),i} - a_{n(k+1),i})_i$ converges in $L^2(B_{n(k)} \cap B_{n(k+1)};\R^N)$.
    As this is a sequence of rigid motions and the intersection $B_{n(k)} \cap B_{n(k+1)}$ is set of positive measure contained in some ball $B(0,R)$ with $R >0$, Lemma \ref{lem_int_rigid} shows that the sequence converges in the normed space of affine maps. It follows that the sequence $(a_{n(k),i} - a_{n(k+1),i})_i$ converges locally uniformly in $\R^N$ to a rigid motion.
    Then, a telescopic argument shows that $(a_i - a_{n,i})_i$ also converges locally uniformly in $\R^N$ to a rigid motion.
    Our claim is proved.
    We deduce that for all $n \geq 0$, $(u_i - a_i)_i$ converges in $L^2(B_n;\R^N)$ to a function in $W^{1,2}(B_n;\R^N)$. Since the balls $(B_n)_n$ cover $\mathcal{O}$, we finally conclude that there exists a function $u \in W^{1,2}_{\mathrm{loc}}(\mathcal{O};\R^N)$ such that for all compact subset $H \subset \mathcal{O}$,
    \begin{equation*}
        \lim_{i \to +\infty} \int_H \abs{u_i - a_i}^2 \dd{x} = 0.
    \end{equation*}
    In this procedure, we have extracted a subsequence of $(u_i)_i$ which depends on $\mathcal{O}$ but as $\Omega \setminus K$ has countably many connected components, we can do a diagonal extraction again so that an analogue property holds for all connected components of $\Omega \setminus K$.
\end{proof}

We now turn our attention to the semi-continuity properties of converging sequence of pairs.
\begin{lemma}[Lower semicontinuity of the elastic energy]\label{lem_semicontinuity_energy}
    Let $(\Omega_i)_i$ and $\Omega$ be a sequence of open sets as in (\ref{eq_Omega}).
    Let $(u_i,K_i)_i$ be a sequence such that for all $i$, $(u_i,K_i)$ is a pair in $\Omega_i$.
    If $(u_i,K_i)_i$ converges to a pair $(u,K)$ in $\Omega$, then for all open set $V \subset \Omega$, and for all real number $p \geq 1$,
    \begin{equation*}
        \int_{V \setminus K} \left[\CC{e(u)}\right]^{p/2} \dd{x} \leq \liminf_{i \to +\infty} \int_{\Omega_i \cap V \setminus K_i} \left[\CC{e(u)}\right]^{p/2} \dd{x}.
    \end{equation*}
\end{lemma}
\begin{proof}
    We start with the case where $V \subset \subset \Omega \setminus K$.
    Observe that the domain $\Omega \setminus K$ can be decomposed as a disjoint union of its connected component and for each component $\mathcal{O}$, we have $V \cap \mathcal{O} \subset \subset \mathcal{O}$ because $\partial \mathcal{O} \cap \Omega \setminus K = \emptyset$.
    Therefore, it suffices to deal with the case where there exists an connected component $\mathcal{O}$ of $\Omega \setminus K$ such that $V \subset \subset \mathcal{O}$.
    We let $(a_i)_i$ be a sequence of rigid motions such that for all compact set $H \subset \mathcal{O}$,
    \begin{equation}\label{eq_ai}
        \lim_{i \to +\infty} \int_H \abs{u_i - a_i} \dd{x} = 0.
    \end{equation}
    We let $\phi \in C_c(V;\R^{N \times N}_{\mathrm{sym}})$ be a smooth test function with compact support in $V$ and which takes its values in $\R^{N \times N}_{\mathrm{sym}}$.
    By integration by parts and (\ref{eq_ai}), one can see that
    \begin{equation*}
        \int \C e(u):\phi \dd{x} = \lim_{i \to +\infty} \int \C e(u_i) : \phi \dd{x}.
    \end{equation*}
    Then, Hölder inequality and the dual representation of norms imply
    \begin{equation*}
        \int_{V} \left[\CC{e(u)}\right]^{p/2} \dd{x} \leq \liminf_{i \to +\infty} \int_{V} \left[\CC{e(u_i)}\right]^{p/2} \dd{x}.
    \end{equation*}

    For a general open set $V \subset \Omega$, we consider an exhaustion of $V \setminus K$ by an increasing sequence of open sets $(V^n)_n$ such that $V^n \subset \subset V \setminus K$. For each $n$ and for $i$ big enough, we have $V^n \subset \Omega_i \setminus K_i$ so
    \begin{align*}
        \int_{V^n} \left[\CC{e(u)}\right]^{p/2} \dd{x} &\leq \liminf_{i \to +\infty} \int_{V^n} \left[\CC{e(u_i)}\right]^{p/2} \dd{x}\\
                                                       &\leq \liminf_{i \to +\infty} \int_{\Omega_i \cap V \setminus K_i} \left[\CC{e(u_i)}\right]^{p/2} \dd{x}
    \end{align*}
    and then by letting $n \to +\infty$,
    \begin{equation*}
        \int_{V \setminus K} \left[\CC{e(u)}\right]^{p/2} \dd{x} \leq  \liminf_{i \to +\infty}\int_{\Omega_i \cap V \setminus K_i} \left[\CC{e(u_i)}\right]^{p/2} \dd{x}.
    \end{equation*}
\end{proof}

For a sequence of sets converging in Hausdorff distance, we don't have the lower semi-continuity of the area in general but we have a rough control if the sequence is uniformly Ahlfors-regular. The limit is in particular, coral and Ahlfor-regular as well. We omit the proof, which is standard.
\begin{lemma}\label{lem_uniformAF}
    Let us fix an open ball $B \subset \R^N$.
    Let $(K_i)_i$ be a sequence of relatively closed subsets of $B$ which converges to a relatively closed subset $K$ of $B$. We assume that there exists a constant $C_0 \geq 1$ such that for all $i$, for all $x \in K_i$, for all $r > 0$ such that $B(x,r) \subset B$, we have
    \begin{equation*}
        C_0^{-1} r^{N-1} \leq \HH^{N-1}(K_i \cap B(x,r)) \leq C_0 r^{N-1}.
    \end{equation*}
    Then, for all open set $V \subset B$ and for all compact set $H \subset B$, we have
    \begin{align*}
        \liminf_{i \to +\infty} \HH^{N-1}(K_i \cap V) &\geq C^{-1} \HH^{N-1}(K \cap V)\\
        \limsup_{i \to +\infty} \HH^{N-1}(K_i \cap H) &\leq C \HH^{N-1}(K \cap H),
    \end{align*}
    for some constant $C \geq 1$ which depends only on $C_0$ and $N$.
\end{lemma}

The last result of this section is a more precise variant of (\ref{eq_KCV}).
\begin{lemma}\label{lem_separation_convergence}
    We consider a sequence of open sets $(\Omega_i)_i \subset \R^N$ and an open set $\Omega$ as in (\ref{eq_Omega}).
    We let $(K_i)_i$ be a sequence such that for all $i$, $K_i$ is a relatively closed subset of $\Omega$ and we assume that $(K_i)_i$ converges to a relatively closed subset $K \subset \Omega$.
    Then if a compact set $H$ is contained in a connected component of $\Omega \setminus K$, it is contained in a connected component of $\Omega_i \setminus K_i$ for $i$ big enough.
\end{lemma}
\begin{proof}
    Let $\mathcal{O}$ be a connected component of $\Omega \setminus K$ and let $H$ be a compact set such that $H \subset \mathcal{O}$. We can cover $H$ by a finite family of balls $B_1,\ldots,B_p$, where $B_h = B(y_h,r_h)$, where $y_h \in H$, $r_h > 0$, such that $\overline{B}_h \subset \mathcal{O}$.
    For all $h_1 \ne h_2$, there exists a continuous path $\gamma \subset \mathcal{O}$ from $y_{h_1}$ to $y_{h_2}$. Since $\gamma$ is a compact subset of $\Omega \setminus K$, it is also contained in $\Omega_i \setminus K_i$ for $i$ big enough.
    Thus for $i$ big enough (depending on $h_1$ and $h_2$), the points $y_{h_1}$ and $y_{h_2}$ lie in a common connected component of $\Omega_i \setminus K_i$.
    If in addition $i$ is also big enough (still depending on $h_1$ and $h_2$) such that $\overline{B}_{h_1}, \overline{B}_{h_2} \subset \Omega_i \setminus K_i$, we deduce that $B_{h_1}$ and $B_{h_2}$ lie in a common connected component of $\Omega_i \setminus K_i$.
    Since there is only a finite number of indices $h = 1,\ldots,p$, we can find $i_0$ such that for all $i \geq i_0$ and for all $h_1 \ne h_2$, the balls $B_{h_1}$ and $B_{h_2}$ lie in a common connected component of $\Omega_i \setminus K_i$.
    Here we see that this connected component cannot depend on $h_1$ and $h_2$ so for $i \geq i_0$, all the balls $B_h$ lie in the same connected component of $\Omega_i \setminus K_i$ and $H$ as well.
\end{proof}

\subsection{A partial limiting property}

Our first step to prove Theorem \ref{thm_limit} is a weaker Proposition which does not rely on the lower semi-continuity along sequences. The goal of the two subsequent sections will be to complete this result by proving the lower semi-continuity.

\begin{proposition}\label{prop_weak_limit}
    Let $(\Omega_i)_i$ and $\Omega$ be a sequence of open sets as in (\ref{eq_Omega}).
    Let $(u_i,K_i)_i$ be a sequence such that for all $i$, $(u_i,K_i)$ is a topological almost-minimizer with gauge $h_i$ in $\Omega_i$. We assume that $(u_i,K_i)_i$ converges to a pair $(u,K)$ in $\Omega$ and we set for $r > 0$,
    \begin{equation*}
        h^+(r) := \lim_{t \to r^+} \limsup_{i \to + \infty} h_i(t).
    \end{equation*}
    Then for all $x \in \Omega$, for all $r > 0$ with $\overline{B}(x,r) \subset \Omega$ and $h^+(r) < \varepsilonA$, for all topological competitor $(v,L)$ of $(u,K)$ in $B(x,r)$, we have
    \begin{multline*}
        \limsup_{i \to +\infty} \left(\int_{B(x,r) \setminus K_i} \CC{e(u_i)} \dd{x} + \HH^{N-1}(K_i \cap \overline{B}(x,r)\right) \\\leq \int_{B(x,r) \setminus L} \CC{e(v)} \dd{x} + \HH^{N-1}(L \cap \overline{B}(x,r)) + h^+(r) r^{N-1}.
    \end{multline*}
    If furthermore
    \begin{equation*}
        \lim_{i \to + \infty} \int_H \abs{e(u_i)} \dd{x} = 0 \quad \text{for all compact set $H \subset \Omega \setminus K$},
    \end{equation*}
    then $u$ is a rigid motion in each connected component of $\Omega \setminus K$. In this case, for all $x \in \Omega$, for all $r > 0$ with $\overline{B}(x,r) \subset \Omega$ and $h^+(r) < \varepsilonA$, for all topological competitor $L$ of $K$ in $B(x,r)$, we have
    \begin{equation*}
        \limsup_{i \to +\infty} \HH^{N-1}(K_i \cap \overline{B}(x,r)) \leq \HH^{N-1}(L \cap \overline{B}(x,r)) + h^+(r) r^{N-1}.
    \end{equation*}
\end{proposition}

The gauge $h^+$ is well-defined because the function $t \mapsto \limsup_{i \to +\infty} h_i(t)$ is non-decreasing on $(0,+\infty)$. One can also see that $h^+$ is right-continuous.
The reason why we work in balls where $h^+(r) < \varepsilonA$ is to ensure that the Ahlfors-regularity properties (\ref{eq_AF}), (\ref{eq_AF2}) hold along the sequence.

\begin{proof}
    We start by focusing on the first part of the statement: the limiting minimality property.
    The general strategy is clear: for a fixed ball $\overline{B}(x,r) \subset \Omega$ and for every topological competitor $(v,L)$ for $(u,K)$ in $B(x,r)$, we need to define a suitable topological competitor $(v_i,L_i)$ for $(u_i,K_i)$ in a slightly larger ball $B(x,r + \delta)$ in order to exploit the minimality of $(u_i,K_i)$ and then pass to the limit.
    For that purpose we will choose a good radius $\rho \in (r,r + \delta)$ satisfying a series of good properties before defining $v_i$ and $L_i$.
    We now fix a ball $B = B(x,r)$ with $r > 0$ such that $\overline{B} \subset \Omega$ and $h(r) < \varepsilonA$.

    \vspace{0.5cm}\noindent
    \emph{Step 1. Construction of the annulus.}
    We let $0 < \delta \leq r$ be so small that
    \begin{equation*}
        \overline{B}(x,r + 10\delta) \subset \Omega \quad \text{and} \quad \limsup_{i \to + \infty} h_i(r + 10\delta) < \varepsilonA
    \end{equation*}
    and in particular,
    \begin{equation}\label{eq_Bdelta}
        \overline{B}(x,r + 10\delta) \subset \Omega_i \quad \text{and} \quad h_i(r + 10\delta) \leq \varepsilonA \quad \text{for $i$ big enough.}
    \end{equation}
    For convenience, we assume that it holds for all $i$.
    In particular, (\ref{eq_Bdelta}) allows to apply (\ref{eq_AF}), (\ref{eq_AF2}), that is, for all $i$, for all open ball $B(y,t) \subset B(x,r + 10\delta)$, we have a uniform bound
    \begin{equation*}
        \int_{B(y,t) \setminus K_i} \abs{e(u_i)}^2 \dd{x} + \HH^{N-1}(K_i \cap B(y,t)) \leq C t^{N-1},
    \end{equation*}
    and if $y \in K_i$,
    \begin{equation*}
        \HH^{N-1}(K_i \cap B(y,t)) \geq C^{-1} \mathrm{diam}(B)^{N-1} 
    \end{equation*}
    for some constant $C \geq 1$ which depends on $N$ and $\C$. According to Lemma \ref{lem_uniformAF}, this implies that for all open ball $B(y,t) \subset B(x,r + 5\delta)$ with $y \in K$, we have
    \begin{equation}\label{eq_KAF}
        C^{-1} t^{N-1} \HH^{N-1}(K \cap B(y,t)) \leq C t^{N-1}.
    \end{equation}

    We let $\tau \in (0,1)$ be a very small parameter which can depend on $r$ and $\delta$ and such that $\tau r \leq \delta$. In what follows, the letter $C$ denotes a generic constant $\geq 1$ which is allowed to depend on $N$, $\C$ and also $r$, $\delta$. We consider a maximal subset $Y \subset K \cap \overline{B}(x,r + 2\delta)$ of points at mutual distance greater than or equal to $\tau r$. Therefore,
    \begin{equation}\label{eq_Y}
        K \cap \overline{B}(x, r + 2\delta) \subset \bigcup_{y \in Y} B(y,\tau r) 
    \end{equation}
    and the balls $B(y,\tau r/2)$, $y \in Y$ are disjoint. We can use (\ref{eq_KAF}) to estimate the number of points of $Y$, denoted by $\abs{Y}$. More precisely,
    \begin{equation*}
        \sum_{y \in Y} \HH^{N-1}(K \cap B(y, \tau r/2)) \geq C^{-1} (\tau r)^{N-1} \abs{Y}
    \end{equation*}
    and since the balls $B(y,\tau r/2)$ are disjoint and contained in $B(x,r + 3\delta)$,
    \begin{equation*}
        \sum_{y \in Y} \HH^{N-1}(K \cap B(y, \tau r/2)) \leq \HH^{N-1}(K \cap B(x,r+3\delta)) \leq C r^{N-1}.
    \end{equation*}
    Hence, $Y$ has at most $C \tau^{1-N}$ points.

    We will choose a suitable annulus of width $\tau r$ which does not intersect too many balls $B(y,\tau r)$, $y \in Y$. More precisely, for $\rho \in (r,r+\delta)$, we let
    \begin{equation*}
        Y_{\rho} := \set{y \in Y | B(y,\tau r) \cap \overline{B}(x,\rho + \tau r) \setminus B(x,\rho) \ne \emptyset}. 
    \end{equation*}
    The condition $y \in Y_{\rho}$ is equivalent to saying that $\abs{y} \in (\rho - \tau r, \rho + 2\tau r)$, or equivalently again, $\rho \in (\abs{y} - 2\tau r, \abs{y} + 2\tau r)$. Then we use Fubini and the fact that $Y$ has at most $C \tau^{1-N}$ points to estimate
    \begin{align*}
        \int_r^{r+\delta} \left(\sum_{y \in Y_{\rho}} \HH^{N-1}(\partial B(y,\tau r))\right) \dd{\rho} &= \sum_{y \in Y} \int_r^{r+\delta} \HH^{N-1}(\partial B(y,\tau r)) \mathbf{1}_{Y_\rho}(y) \dd{\rho}\\
                                                                                                       &= \sum_{y \in Y} \int_r^{r+\delta} \HH^{N-1}(\partial B(y,\tau r)) \mathbf{1}_{(\abs{y} - 2\tau r, \abs{y} + \tau r)}(\rho) \dd{\rho}\\
                                                                                                       &\leq C \tau,
    \end{align*}
    where we recall that $C$ is allowed to depend on $r$ and $\delta$. This implies that there are many $\rho \in (r,r+\delta)$ such that
    \begin{equation*}
        \sum_{y \in Y_\rho} \HH^{N-1}(\partial B(y,\tau r)) \leq C \tau, 
    \end{equation*}
    where $C$ is a bigger constant which is still allowed to depend on $r$ and $\delta$. Let us choose such a radius $\rho \in (r,r+\delta)$ and define
    \begin{equation*}
        Z := \bigcup_{y \in Y_\rho} B(y, \tau r). 
    \end{equation*}
    Then we consider a smooth scalar cut-off function $\varphi \in C_c^{\infty}(\R^N,[0,1])$ such that $0 \leq \varphi \leq 1$,
    $$\varphi= 1 \ \ \text{in} \ \ {B}(x,\rho), \quad \quad \varphi = 0 \ \ \text{in} \ \ \R^N \setminus B(x,\rho + \tau r/2),$$ and
    $$|\nabla \varphi| \leq C \tau^{-1} \ \ \text{everywhere.}$$
    Let us finally define 
    \begin{equation*}
        Z':= \bigcup_{y \in Y_\rho} \partial B(y,\tau r),
    \end{equation*}
    so that $\partial Z \subset Z'$ and
    \begin{equation*}
        \mathcal{H}^{N-1}(Z')\leq C \tau.
    \end{equation*}
    By construction, $Z' \subset B(x,r + 4\tau r)$.

    \vspace{0.5cm}\noindent
    \emph{Step 2. Construction of the competitor.} We now proceed to build a competitor $(v_i,L_i)$ of $(u_i,K_i)$ in $B(x,r + 4\tau r)$ which makes a transition between $(u_i,K_i)$ outside of $B(x,\rho + \tau r)$ and $(v,L)$ in $B(x,\rho)$.
    First of all, we observe that $Z$ covers the sets $K$, $K_i$, $L$ in the transition area.
    More precisely, we see from (\ref{eq_Y}) and the definition of $Y_\rho$ that
    \begin{equation}\label{eq_KZ}
        K \cap \overline{B}(x,\rho + \tau r) \setminus B(x,\rho) \subset Z,
    \end{equation}
    and thus, by convergence of $(K_i)_i$ to $K$,
    \begin{equation}\label{eq_KIZ}
        K_i \cap \overline{B}(x,\rho + \tau r) \setminus B(x,\rho) \subset Z \quad \text{for $i$ big enough}. 
    \end{equation}
    By the fact that $L$ coincides with $K$ outside of $B(x,r)$, we also have
    \begin{equation}\label{eq_LZ}
        L \cap \overline{B}(x,\rho + \tau r) \setminus B(x,\rho) \subset Z. 
    \end{equation}
    For convenience, we assume that (\ref{eq_KIZ}) holds for all $i$. We now define
    \begin{equation*}
        L_i := (K_i \setminus B(x,\rho)) \cup Z' \cup (L \cap \overline{B}(x,\rho)), 
    \end{equation*}
    which is a relatively closed subset of $\Omega_i$ and coincides with $K_i$ in $\Omega_i \setminus B(x,r + 4 \tau r)$. We then define $v_i$ in a piecewise way. We first set
    \begin{equation}\label{eq_vi0}
        v = 0 \ \ \text{in} \ \ Z \quad \text{and} \quad v = u_i \ \ \text{in} \ \ \Omega_i \setminus \left(B(x,\rho + \tau r) \cup L_i \cup \overline{Z}\right). 
    \end{equation}
    Then we build a transition between $u_i$ and $v$ (up to a suitable rigid motion) in $B(x,\rho + \tau r) \setminus \left(L_i \cup \overline{Z}\right)$. By (\ref{eq_KZ}) and (\ref{eq_KIZ}), we see that the annulus
    \begin{equation*}
        \overline{B}(x,\rho + \tau r) \setminus \left(B(x,\rho) \cup Z\right) 
    \end{equation*}
    is a compact subset of $\Omega \setminus K$ and $\Omega_i \setminus K_i$ (in particular, $u_i$ is well-defined there). By compactness, this annulus is covered by a finite number of connected components $\mathcal{O}_1,\ldots,\mathcal{O}_p$ of $\Omega \setminus K$. For each $\ell = 1,\ldots,p$, there exists a sequence of rigid motions $(a_{i,\ell})_i$ such that for all compact set $H \subset \mathcal{O}_{\ell}$,
    \begin{equation}\label{eq_ail}
        \text{the sequence $(u_i - a_{i,\ell})_i$ converges in $L^2$ norm to $u$ on $H$.}
    \end{equation}
    The compact sets we have in mind are the sets
    \begin{equation*}
        H_{\ell} := \mathcal{O}_{\ell} \cap \overline{B}(x,\rho + \tau r) \setminus \left(B(x,\rho) \cup Z\right).
    \end{equation*}
    Indeed, since $\mathcal{O}_{\ell} \cap \Omega \setminus K = \emptyset$ and $\overline{B}(x,\rho + \tau r) \setminus \left(B(x,\rho) \cup Z\right) \subset \Omega \setminus K$, the set $H_{\ell}$ is a compact subset of $\mathcal{O}_{\ell}$.

    Let us now consider a connected component $V$ of $B(x,\rho + \tau r) \setminus \left(L_i \cup \overline{Z}\right)$. If $V \subset B(x,\rho)$, there is no need to make a transition and we just set $v_i = 0$. Otherwise, $V \cap B(x,\rho + \tau r) \setminus B(x,\rho)$ is non-empty and we are going to check that there exists a (necessarily unique) $\ell = 1,\ldots,p$ such that
    \begin{equation}\label{eq_VOL}
        V \cap B(x,\rho + \tau r) \setminus B(x,\rho) \subset \mathcal{O}_{\ell}. 
    \end{equation}
    Let $x,y \in V \cap B(x,\rho + \tau r) \setminus B(x,\rho)$. By (\ref{eq_LZ}), we have
    \begin{equation*}
        L \cap B(x,\rho + \tau r) \setminus B(x,\rho) \subset Z 
    \end{equation*}
    and by definition of $L_i$,
    \begin{equation*}
        L \cap \overline{B}(x,\rho) \subset L_i
    \end{equation*}
    so $V$, as a connected component of $B(x,\rho + \tau r) \setminus \left(L_i \cup \overline{Z}\right)$, is disjoint from $L$. This shows that $x$ and $y$ lie in the same connected component of $\Omega \setminus L$ and since $L$ is a topological competitor of $K$ in $B(x,r)$, they also lie in the same connected component of $\Omega \setminus K$. This proves (\ref{eq_VOL}) and this leads us to set
    \begin{equation}\label{eq_vi}
        v_i = \varphi(v + a_{i,\ell}) + (1 - \varphi) u_i \quad \text{in $V$}. 
    \end{equation}
    This achieves the definition of $v_i$ in $B(x,\rho + \tau r) \setminus \left(L_i \cup \overline{Z}\right)$. Combining (\ref{eq_vi0}) and (\ref{eq_vi}), we see that
    \begin{equation*}
        v_i = u_i \ \ \text{in} \ \ \Omega_i \setminus \left(B(x,\rho + \tau r/2) \cup L_i \cup \overline{Z}\right) 
    \end{equation*}
    so there is no gluing problem along $\partial B(x,\rho + \tau r) \setminus \left(L_i \cup \overline{Z}\right)$. We conclude that $(v_i,L_i)$ is a competitor of $(u_i,K_i)$ in $B(x,r + 4 \tau r)$.
    We now check that $L_i$ is a topological competitor of $K_i$ in $B(x,r + 4\tau r)$, for $i$ big enough.
    First of all, we recall for each $\ell = 1,\ldots,p$, the set
    \begin{equation*}
        H_{\ell} = \mathcal{O}_{\ell} \cap \overline{B}(x,\rho + \tau r) \setminus \left(B(x,\rho) \cup Z\right)
    \end{equation*}
    is a compact of $\mathcal{O}_{\ell}$. Using Lemma \ref{lem_separation_convergence} and since there are only a finite number of indices $\ell = 1,\ldots,p$, we can find an index $i_0$ such that for all $i \geq i_0$ and for all $\ell = 1,\ldots,p$,
    \begin{equation}\label{eq_HIO}
        \text{the set $H_{\ell}$ is contained in a connected component of $\Omega_i \setminus K_i$.}
    \end{equation}
    Let us now consider $i \geq i_0$.
    We fix $y,z \in \Omega_i \setminus \left(B(x,r + 4\tau r) \setminus K_i\right)$ such that $y,z$ are connected by a continuous path $\gamma :[0,1] \to \Omega_i \setminus L_i$ and we prove that they are connected in $\Omega_i \setminus K_i$.
    We proceed by contradiction and assume that $y$ and $z$ lie in distinct connected components of $\Omega_i \setminus K_i$.
    We first observe that $y,z \notin \overline{Z}$ since $\overline{Z} \subset B(x,r + 4 \tau r)$. As the path $\gamma$ is disjoint from $L_i$, it is in particular disjoint from $Z'$ and therefore it must stay disjoint from $\overline{Z}$. If $\gamma$ never meets $B(x,\rho)$, then $\gamma$ is disjoint from $K_i$ because $L_i$ coincides with $K_i$ outside of $B(x,\rho) \cap \overline{Z}$.
    In this case, $y,z$ are connected by $\gamma$ in $\Omega_i \setminus K_i$ and we reach a contradiction.
    Next, we assume that $\gamma$ meets $B(x,\rho)$ and we let $y_1$, $z_1$ denote the first and last point of $\gamma$ on $\partial B(x,\rho)$. On the portion between $y$ and $y_1$, $\gamma$ lies in the complement of $B(x,\rho)$ and then one can deduce as before that this portion lies in the complement of $K_i$. Therefore $y$ and $y_1$ are connected in the $\Omega_i \setminus K_i$. Similarly, $z$ and $z_1$ are connected in $\Omega_i \setminus K_i$. It follows that $y_1$ and $z_1$ lie in distinct connected components of $\Omega_i \setminus K_i$. 
    The set of points of $\gamma \cap \partial B(x,\rho)$ which do not lie in the connected component of $\Omega_i \setminus K_i$ containing $z_1$ is non-empty (it contains $x_1)$ and closed. Therefore, there is a last point of $\gamma$ in this set, and we let it be denoted by $y_2$.
    Observe that $y_2 \notin K_i$ because $\gamma$ is disjoint from $\overline{Z}$ and because of (\ref{eq_KIZ}).
    Then, we let $z_2$ be the first of $\gamma \cap \partial B$ after $y_2$. Here again, $z_2 \notin K_i$ for the same reason.
    We deduce that $y_2$ and $z_2$ lie in distinct connected components of $\Omega_i \setminus K_i$.
    The portion of $\gamma$ between $y_2$ and $z_2$ does not meet $\partial B(x,\rho)$ to it must be either in $\Omega_i \setminus \overline{B}(x,\rho)$ or in $B(x,\rho)$. In the first case, $y_2$ and $z_2$ are connected in $\Omega_i \setminus K_i$ so we reach a contradiction. In the second case, $y_2$ and $z_2$ are connected in $\Omega \setminus L$ and since $L$ is a topological competitor of $K$ in $B(x,r)$, the points $y_2$ and $z_2$ must also be connected in $\Omega \setminus K$. We deduce that $y_2$ and $z_2$ lie in a common set $H_{\ell}$ for some $\ell = 1,\ldots,p$ and thus in a common connected component of $\Omega_i \setminus K_i$ by (\ref{eq_HIO}). This is again a contradiction.
    We have proved that for all $i \geq i_0$, $L_i$ is a topological competitor of $K_i$ in $B(x,r + 4 \tau r)$.

    \vspace{0.5cm}\noindent
    \emph{Step 3. Energy comparison.}
    We finally apply the almost minimality property of $(u_i,K_i)$ and compare its Griffith energy with $(v_i,L_i)$,
    \begin{multline*}
        \int_{B(x,r + 4\tau r) \setminus K_i} \CC{e(u_i)} \dd{x} + \HH^{N-1}(K_i \cap B(x,r + 4\tau r) \\\leq \int_{B(x,r + 4\tau r) \setminus L_i} \CC{e(v_i)} \dd{x} + \HH^{N-1}(L_i \cap B(x,r + 4\tau r)) + h_i(r + 4\tau r) (r + 4\tau r)^{N-1}.
    \end{multline*}
    Using the facts that $L_i \setminus B(x,\rho) \subset (K_i \setminus B(x,\rho)) \cup Z'$, that $\HH^{N-1}(Z') \leq C \tau$ and that $\abs{e(v_i)} \leq \abs{e(u_i)}$ a.e. in $\Omega_i \setminus B(x,\rho + \tau r)$, we arrive at
    \begin{multline}\label{eq_comparison_vi}
        \int_{B(x,\rho + \tau r) \setminus K_i} \CC{e(u_i)} \dd{x} + \HH^{N-1}(K_i \cap B(x,\rho)) \\\leq \int_{B(x,\rho + \tau r) \setminus L_i} \CC{e(v_i)} \dd{x} + \HH^{N-1}(L \cap B(x,\rho)) + C \tau + h_i(r + 4\tau r) (r + 4 \tau r)^{N-1}.
    \end{multline}
    We now estimate the contribution of $e(v_i)$ in $B(x,\rho + \tau r) \setminus L_i$. The points in $B(x,\rho + \tau r) \setminus L_i$ are either contained in $Z$, where $e(v_i) = 0$, or in a connected component $V$ of $B(x,\rho + \tau r) \setminus \left(L_i \cup \overline{Z}\right)$ such that
    \begin{equation}\label{eq_VOL2}
        V \cap B(x,\rho + \tau r) \setminus B(x,\rho) \subset \mathcal{O}_{\ell}
    \end{equation}
    for some $\ell = 1,\ldots,p$ (see (\ref{eq_VOL})), and where
    \begin{equation*}
        e(v_i) = \varphi e(v) + (1 - \varphi) e(u_i) + \nabla \varphi \odot (v + a_{i,\ell} - u_i). 
    \end{equation*}
    Here, given $a,b \in \R^N$, the notation $a \odot b$ denotes the matrix of coefficients
    \begin{equation*}
        (a \odot b)_{ij} = \frac{a_i b_j + a_j b_i}{2} \in \R^{N \times N}.
    \end{equation*}
    Note that one can bound $\abs{a \odot b} \leq \abs{a} \abs{b}$.
    The function $\xi \mapsto \C \xi : \xi$ is a positive definite quadratic form on the space $\R^{N \times N}_{\rm sym}$ of symmetric matrices and it is temporarily convenient to work with the underlying norm. We set
    \begin{equation*}
        \NC{\xi} := \sqrt{\C \xi: \xi} \quad \text{for} \ \xi \in \R^{N \times N}_{\rm sym}.
    \end{equation*}
    In a connected component $V$ of $B(x,\rho + \tau r) \setminus \left(L_i \cup \overline{Z}\right)$ where (\ref{eq_VOL2}) holds, we have by triangular inequality
    \begin{equation*}
        \NC{e(v_i)} \leq \varphi \NC{e(v)} + (1 - \varphi) \NC{e(u_i)} + C \abs{\nabla \varphi} \abs{u_i - a_{i,\ell} - v}. 
    \end{equation*}
    The function $\nabla \varphi$ is supported in $B(x,\rho + \tau r) \setminus B(x,\rho)$, satisfies $\abs{\nabla \varphi} \leq C \tau^{-1}$ and we see from (\ref{eq_VOL2}) that
    \begin{equation*}
        V \cap B(x,\rho + \tau r) \setminus B(x,\rho) \subset H_\ell,
    \end{equation*}
    where
    \begin{equation*}
        H_{\ell} = \mathcal{O}_{\ell} \cap \overline{B}(x,\rho + \tau r) \setminus \left(B(x,\rho) \cup Z\right)
    \end{equation*}
    is a compact subset of $\mathcal{O}_{\ell}$.
    Thus we can bound in $B(x,\rho + \tau r)$,
    \begin{equation}\label{eq_normvi}
        \NC{e(v_i)} \leq \varphi \NC{e(v)} + (1 - \varphi) \NC{e(u_i)} + C \tau^{-1} \sum_{\ell = 1,\ldots,p} \abs{u_i - a_{i,\ell} - v} \mathbf{1}_{H_{\ell}}.
    \end{equation}
    We will be able to get rid of the last term when $i \to + \infty$ because we know from (\ref{eq_ail}) that the sequence $(u_i - a_{i,\ell})_i$ converges in $L^2$ norm to $v$ on $H_{\ell}$.
    Let us estimate the $L^2$ norm of $e(v)$ in $B(x,\rho + \tau r)$. We use (\ref{eq_normvi}), the elementary inequality
    \begin{equation*}
        (a + b_1 + \ldots + b_p)^2 \leq (1 + \varepsilon) a^2 + p (1 + \varepsilon^{-1}) \left(b_1^2 + \ldots + b_p^2\right) \quad \text{for all $\varepsilon > 0$} 
    \end{equation*}
    and the convexity of $t \mapsto t^2$ to bound
    \begin{align*}
        \begin{split}
            \int_{B(x,\rho + \tau r)} \NC{e(v_i)}^2 \dd{x} &\leq (1 + \varepsilon) \int_{B(x,\rho + \tau r)} \left(\varphi \NC{e(v)} + (1 - \varphi) \NC{e(u_i)}\right)^2 \dd{x}\\
                                                             &\qquad + C(p) \tau^{-2} (1 + \varepsilon^{-1}) \sum_{\ell=1,\ldots,p} \int_{H_{\ell}} \abs{u_i - a_{i,\ell} - v}^2 \dd{x}
        \end{split}\\
        \begin{split}
        &\leq (1 + \varepsilon) \int_{B(x,\rho + \tau r)} \varphi \NC{e(v)}^2 + (1 - \varphi) \NC{e(u_i)}^2 \dd{x}\\
        &\qquad + C(p) \tau^{-2} (1 + \varepsilon^{-1}) \sum_{\ell = 1,\ldots,p} \int_{H_{\ell}} \abs{u_i - a_{i,\ell} - v}^2 \dd{x}.
        \end{split}
    \end{align*}
    Plugging this in (\ref{eq_comparison_vi}), we arrive at
    \begin{multline}\label{eq_comparison_vi2}
        \int_{B(x,\rho + \tau r) \setminus K_i} \varphi \NC{e(u_i)}^2 \dd{x} + \HH^{N-1}(K_i \cap B(x,\rho)) \\\leq (1 + \varepsilon) \int_{B(x,\rho + \tau r) \setminus L} \NC{e(v)}^2 \dd{x} + \varepsilon \int_{B(x,\rho + \tau r)} \NC{e(u_i)}^2 \dd{x}\\
        + C(p) \tau^{-2} (1 + \varepsilon^{-1}) \sum_{\ell=1,\ldots,p} \int_{H_{\ell}} \abs{u_i - a_{i,\ell} - v}^2 \dd{x}\\
        + \HH^{N-1}(L \cap B(x,\rho)) + C \tau + h_i(r + 4\tau r) (r + 4\tau r)^{N-1}.
    \end{multline}
    We recall that by (\ref{eq_Bdelta}), we can bound $\int_{B(x,\rho + \tau r)} \NC{e(u_i)}^2 \dd{x} \leq C$.
    We use this bound, we come back to the notation $\C \xi:\xi$ and we pass to the limit $i \to + \infty$ in (\ref{eq_comparison_vi2}) to obtain
    \begin{multline*}
        \limsup_{i \to +\infty} \left(\int_{B(x,r) \setminus K_i} \CC{e(u_i)} \dd{x} + \HH^{N-1}(K_i \cap \overline{B}(x,r))\right) \\\leq (1 + \varepsilon) \int_{B(x,\rho + \tau r) \setminus L} \CC{e(v)} \dd{x} + C \varepsilon\\ + \HH^{N-1}(L \cap B(x,\rho)) + C \tau + \limsup_{i \to + \infty} h_i(r + 4\tau r) (r + 4\tau r)^{N-1}.
    \end{multline*}
    Then we let $\varepsilon \to 0$ and then $\tau \to 0$ to conclude
    \begin{multline*}
        \limsup_{i \to +\infty} \left(\int_{B(x,r) \setminus K_i} \CC{e(u_i)} \dd{x} + \HH^{N-1}(K_i \cap \overline{B}(x,r))\right) \\\leq \int_{B(x,r) \setminus L} \CC{e(v)} \dd{x} + \HH^{N-1}(L \cap \overline{B}(x,r)) + h^+(r) r^{N-1}.
    \end{multline*}

    \vspace{0.5cm}\noindent
    \emph{Step 4. The case of a vanishing elastic energy.}
    We pass to the last part of the statement. We assume that
    \begin{equation*}
        \lim_{i \to +\infty} \int_H \abs{e(u_i)} \dd{x} = 0 \quad \text{for all compact set $H \subset \Omega \setminus K$},
    \end{equation*}
    and we prove that the limit satisfies a simplified minimality condition.
    By Lemma \ref{lem_semicontinuity_energy}, we know that $e(u) = 0$ almost-everywhere on $\Omega \setminus K$.
    Therefore, for each connected component $\mathcal{O}_{\ell} $ of $\Omega \setminus K$, there exists a rigid motion $a_{\ell}$ such that $u = a_\ell$ a.e. in $\mathcal{O}_\ell$. Now, let $L$ be a topological competitor of $K$ in some ball $B(x,r) \subset \subset \Omega$.
    We are going to define a suitable function $v \in W^{1,2}_{\mathrm{loc}}(\Omega \setminus L;\R^N)$ such that $v = u$ a.e. in $\Omega \setminus B(x,r)$.
    For each connected component $V$ of $\Omega \setminus L$, we distinguish two cases. If $V \setminus \overline{B}(x,r) \ne \emptyset$,  then there exists a unique connected component $\mathcal{O}_{\ell}$ of $\Omega \setminus K$ such that $V \setminus \overline{B}(x,r) \subset \mathcal{O}_{\ell}$. Indeed, the points of $V \setminus \overline{B}(x,r)$ belong to the same connected components of $\Omega \setminus L$, so they also belong to the same connected component of $\Omega \setminus K$. As $V \setminus \overline{B}(x,r)$ is non-empty, this connected component must be unique. In this case, we set $v = a_\ell$ in $V$ and we note we have $v = u$ a.e. in $V \setminus \overline{B}(x,r) \subset \mathcal{O}_{\ell}$.
    If on the other hand, $V \subset \overline{B}(x,r)$, then we just set $v = 0$ inside $V$ and this is compatible with the Dirichlet condition on $v$. Since both $u$ and $v$ are piecewise rigid, only the surface terms are involved in the energy comparison.
\end{proof}

\section{Fine lower density bound for quasiminimizers}\label{section_mainHole}

The main goal of this section is to prove the following proposition.
We work in the general setting of quasiminimizers as the statements of this section have an independent interest.

\begin{proposition}\label{prop_mainHole}
    For each $M \geq 1$ and $p \in (2(N-1)/N,2]$, there exists $\varepsilon_0 > 0$ (depending on $N$, $\C$, $M$, $p$) and for all $\varepsilon \in (0,1)$, there exists $\varepsilon_1 > 0$ (depending on $N$, $\C$, $M$, $p$, $\varepsilon$) such that the following holds.
    Let $(u,K)$ be a topological $M$-quasiminimizer with gauge $h$ in $\Omega$. For all $x_0 \in K$, for all $r_0 > 0$ with $B(x_0,r_0) \subset \Omega$ and $h(r_0) \leq \varepsilon_0$, if
    \begin{equation*}
        \beta(x_0,r_0) + \omega_p(x_0,r_0) \leq \varepsilon_1
    \end{equation*}
    then we have
    \begin{equation*}
        \HH^{N-1}(K \cap B(x_0,r_0)) \geq (1 - \varepsilon) \omega_{N-1} r_0^{N-1},
    \end{equation*}
    where $\omega_{N-1}$ is the measure of the $(N-1)$-dimensional unit disk.
\end{proposition}

The proof will need several preliminary lemmas that we write below.

\subsection{Initialization of the jump}\label{section_jump}

We define the ``normalized jump'' similarly to \cite{DavidBOOK}.
Let $(u,K)$ be a pair in $\Omega$. Let $x_0 \in K$, $r_0 > 0$ such that $B(x_0,r_0) \subset \Omega$ and $\beta_K(x_0,r_0) \leq 1/2$. We choose a hyperplane $P_0$ which achieves the infimum in the definition of $\beta(x_0,r_0)$ and we choose a unit normal $\nu_0$ to $P$. We define $a_1, a_2$ as the two rigid motions that approximate $u$ in the lower and upper part of $B(x_0,r_0)$, namely for $i = 1,2$,
\begin{equation*}
    a_i(x) = b_i + A_i(x - x_0)
\end{equation*}
where $b_i \in \R^N$ and $A_i \in \R^{N \times N}$ are such that
\begin{equation}\label{eq_average_rigid}
    b_i = \fint_{D_i} u(y) \dd{y}, \quad A_i = \fint_{D_i} \frac{\nabla u(y) - \nabla u(y)^T}{2} \dd{y}
\end{equation}
and $D_i \subset \subset B(x_0,r_0) \setminus K$ is the domain defined by
\begin{equation}\label{eq_domains}
    D_1 := B(x_0 + (3r_0/4) \nu_0,r_0/8), \quad D_2 := B(x_0 - (3r_0/4) \nu_0,r_0/8).
\end{equation}
Then, we define the \emph{normalized jump} of $u$ in $B(x_0,r_0)$ as
\begin{equation}\label{eq_normalized_jump}
    J(x_0,r_0) := \frac{\abs{b_1 - b_2} + r_0 |A_1 - A_2|}{\sqrt{r_0}}.
\end{equation}
This quantity is invariant under rescaling, see Remark \ref{rmk_scaling}.
We also recall the definition of the $p$-normalized elastic energy, defined for $p \geq 1$ by
\begin{equation*}
    \omega_p(x_0,r_0):= r_0^{1-2N/p}\left(\int_{B(x_0,r_0)\setminus K} \abs{e(u)}^p \dd{x}\right)^{\frac{2}{p}}.
\end{equation*}

A classical argument in \cite{DavidBOOK}, which is also a first step toward the proof of Proposition \ref{prop_mainHole}, says that when $\beta$ and $\omega$ are small enough then $J$ is bounded from below. We are going to adapt the argument to Griffith quasiminimizers.
We first recall a basic estimate about the harmonic extension  from a sphere to the ball. The proof is given in \cite[Lemma 22.32]{DavidBOOK}.
\begin{lemma}(Estimate about an extension \cite[Lemma 22.32]{DavidBOOK})\label{lem_harmonic}
    For each $p\in (2(N-1)/N,2]$, there is a constant $C \geq 1$ (which depends on $N$ and $p$) such that if $B=B(x,r)$ is a ball in $\R^N$ and $f\in W^{1,p}(\partial B)$ then there is a function $v\in W^{1,2}(B)$ such that 
    $$\int_{B} \abs{\nabla v}^2 \dd{x} \leq C r^{N-\frac{2N}{p}+\frac{2}{p}} \left(\int_{\partial B} |\nabla f|^p\right)^{\frac{2}{p}}$$
    and $v$ has a trace on $\partial B$ that coincide with $f$ almost-everywhere.
\end{lemma}

\begin{lemma}(Initialization of the jump)\label{lem_Jinit}
    For each $M \geq 1$ and $p\in (2(N-1)/N,2]$, there exists a constant $\tau_0 > 0$ (depending on $N$, $\C$, $M$ and $p$) such that the following holds.
    Let $(u,K)$ be a topological $M$-quasiminimizer with gauge $h$ in $\Omega$. For all $x_0 \in K$, for all $r_0 > 0$ such that $B(x_0,r_0) \subset \Omega$,  
    \begin{equation*}
        \beta_K(x_0,r_0) + \omega_p(x_0,r_0) + h(r_0) \leq \tau_0,
    \end{equation*}
    and
    \begin{equation}\label{eq_Jtopo}
        \text{$D_1$ and $D_2$ lie in the same connected component of $\Omega \setminus K$},
    \end{equation}
    where $D_1$ and $D_2$ are the domains defined in (\ref{eq_domains}), then we have
    \begin{equation*}
        J(x_0,r_0) \geq \tau_0.
    \end{equation*}
\end{lemma}
The proof is similar to that of \cite[Proposition 42.10]{DavidBOOK}. We proceed by contradiction and by assuming $\beta + \omega + J^{-1} + h \ll 1$, one build a better competitor of $u$ by removing $K \cap B(x_0,r_0)$ and making an interpolation between the two rigid motions $a_1$ and $a_2$. The quantity $J$ estimates the cost of such an interpolation.
The assumption (\ref{eq_Jtopo}) ensures that when we remove a piece of $K$ in $B(x_0,r_0)$, we still have a topological competitor.
Note that if $(u,K)$ is a plain quasiminimizer (without the topological constraint \eqref{eq_topo_competitor} on competitors), the assumption (\ref{eq_Jtopo}) is not needed.
\begin{proof}
    We let the letter $C$ denotes a constant $\geq 1$ which depends only on $N$, $\C$, $M$ and $p$.
    Since the statement is invariant under rescaling, we can assume that $B(x_0,r_0) = B(0,1)$ and we choose a system of coordinates such that the infimum in the definition of the flatness is achieved for $P_0 = \set{x_N = 0}$.
    Let $(u,K)$ be a topological $M$-quasiminimizer with gauge $h$ in $B(0,1)$.
    We let $\varepsilon_0, \varepsilon_1, \varepsilon_2 \in (0,1/10)$ be a small parameter such that
    \begin{equation}\label{eq_Jhyp}
        \beta_K(0,1) \leq \varepsilon_0,\quad \omega_p(0,1) \leq \varepsilon_1,\quad h(1) \leq \varepsilon_2.
    \end{equation}
    We let $a_1(x) = b_1 + A_1 x$ and $a_2(x) = b_2 + A_2 x$ be two rigid motions approximating $u$ in the upper and lower part of $B(0,1)$, as defined in (\ref{eq_average_rigid}).
    We introduce the open ring $R := B(0,1) \setminus \overline{B}(0,3/4)$
    According to  Korn-Poincaré inequality in the domains $\set{x \in R | \pm x_N > \varepsilon_0}$ (which are disjoint from $K$), we have
    \begin{equation}\label{eq_JKorn}
        \int_{R \cap \set{x_N > \varepsilon_0}} \abs{u - a_1}^p + \abs{\nabla u - \nabla a_1}^p \dd{x} \leq C \omega_p(0,1)^{p/2}
    \end{equation}
    and similarly
    \begin{equation}\label{eq_JKorn2}
        \int_{R \cap \set{x_N < -\varepsilon_0}} \abs{u - a_2}^p + \abs{\nabla u - \nabla a_2}^p \dd{x} \leq C \omega_p(0,1)^{p/2}.
    \end{equation}
    Note that the constant $C$ here is independent of $\varepsilon_0 \in (0,1/10)$.
    We start by building an interpolation of these two rigid motions in the ring $R$.
    We consider a function $\varphi_0 \in C^{\infty}(R)$ such that $0 \leq \varphi_0 \leq 1$, $\abs{\nabla \varphi_0} \leq C \varepsilon_0^{-1}$ and
    \begin{equation*}
        \text{$\varphi_0 = 1$ in $R \cap \set{x_N > \varepsilon_0}$}, \qquad \text{$\varphi_0 = 0$ in $R \cap \set{x_N < -\varepsilon_0}$}.
    \end{equation*}
    Then we let $\bar{u} : R \to \R^N$ be defined by
    \begin{equation*}
        \bar{u}(x) = \varphi_0(x) a_1(x) + (1 - \varphi_0(x)) a_2(x).
    \end{equation*}
    We compute
    \begin{equation}
        e(\bar{u}) = \nabla \varphi_0(x) \odot (a_1(x) - a_2(x))
    \end{equation}
    and we observe that the elastic energy of such an interpolation is controlled by $J$, namely,
    \begin{equation*}
        \int_{R} \abs{e(\bar{u})}^2 \dd{x} \leq C \varepsilon_0^{-2} J(0,1)^2.
    \end{equation*}
    The inequalities (\ref{eq_JKorn}) and (\ref{eq_JKorn2}) can be reformulated as
    \begin{equation*}
        \int_{R \cap \set{\abs{x_N} > \varepsilon_0}} \abs{u - \bar{u}}^p + \abs{\nabla u - \nabla \bar{u}}^p \dd{x} \leq C \omega_p(0,1)^{p/2}
    \end{equation*}
    and this allows to select a radius $\rho \in (3/4,1)$ such that
    \begin{equation*}
        \begin{gathered}
        \text{$u - \bar{u} \in W^{1,2}(\partial B(0,\rho) \cap \set{\abs{x_N} > \varepsilon_0};\R^N)$}\\
        \text{with a tangential derivative given by the restriction of  $\nabla u - \nabla \bar{u}$,}
        \end{gathered}
    \end{equation*}
    and
    \begin{equation}\label{eq_Jcoarea}
        \int_{\partial B(0,\rho) \cap \set{\abs{x_N} > \varepsilon_0}} \abs{u - \bar{u}}^p + \abs{\nabla u - \nabla \bar{u}}^p \dd{\HH^{N-1}} \leq C \omega_p(0,1)^{p/2}.
    \end{equation}
    Then, we make an extension of $u - \bar{u}$ from $\partial B(0,\rho) \cap \set{\abs{x_N} > \varepsilon_0}$ to the whole sphere $\partial B(0,\rho)$.
    We set $B := B(0,\rho)$ and we consider a function $\varphi \in C^1(\partial B)$ such that $0 \leq \varphi \leq 1$, 
    \begin{equation*}
            \text{$\varphi = 1$ in $\set{x \in \partial B | \abs{x_N} > 3 \varepsilon_0}$},\quad \text{$\varphi = 0$ in $\set{x \in \partial B | \abs{x_N} < 2\varepsilon_0}$}
    \end{equation*}
    and $\abs{\nabla \varphi} \leq C \varepsilon_0^{-1}$. We finally define $f(x) := \varphi(x) (u(x) - \bar{u}(x)) \in W^{1,2}(\partial B;\R^N)$.
    We have
    \begin{equation*}
        \abs{\nabla f} \leq \abs{\varphi} \abs{\nabla u - \nabla \bar{u}} + \abs{\nabla \varphi} \abs{u - \bar{u}}
    \end{equation*}
    and by (\ref{eq_Jcoarea}) and the facts that $\abs{\varphi} \leq 1$ and $\abs{\nabla \varphi} \leq C \varepsilon_0^{-1}$, we can estimate
    \begin{equation*}
        \int_{\partial B} \abs{\nabla f}^p \dd{\HH^{N-1}}  \leq C \int_{\partial B(0,\rho) \cap \set{\abs{x_N} > \varepsilon_0}} \abs{u - \bar{u}}^p + \abs{\nabla u - \nabla \bar{u}}^p \dd{\HH^{N-1}} \leq C \varepsilon_0^{-p} \omega_p(0,1)^{p/2}.
    \end{equation*}
    Then by Lemma \ref{lem_harmonic}, there exists a function $v \in W^{1,2}(B;\R^N)$ with a trace which coincides with $f$ almost-everywhere on $\partial B$ such that
    \begin{equation*}
        \int_B \abs{\nabla v}^2 \dd{x} \leq C \left(\int_{\partial B} \abs{\nabla f}^p\right)^{2/p} \leq C \varepsilon_0^{-2} \omega_p(0,1).
    \end{equation*}
    We finally define a competitor $(u^*,K^*)$ of $(u,K)$ in $B(0,1)$ by
    \begin{equation*}
        K^* := \left(K \setminus B(0,\rho)\right) \cup Z,
    \end{equation*}
    where $Z := \set{x \in \partial B(0,\rho) | \abs{x_N} \leq 3\varepsilon_0}$, and
    \begin{equation*}
        u^* := \begin{cases}
            v(y) + \bar{u}(y) &\text{in $B(0,\rho)$}\\
            u(y)        &\text{in $\Omega \setminus \left(B(0,\rho) \cup Z\right)$.}
        \end{cases}
    \end{equation*}
    Remember that $f = u(y) - \bar{u}(y)$ on $\partial B(0,\rho) \setminus Z$ so the two functions glue well along $\partial B(0,\rho) \setminus Z$.
    We also need to check that it satisfies the topological condition \eqref{eq_topo_competitor}, i.e., that all $x,y \in \Omega \setminus (K \cup \overline{B}(0,\rho))$ which are not separated by $K^*$, are not separated by $K$ either.
    So let $\gamma$ be a continuous path connecting $x,y$ in the complement of $K^*$. If $\gamma$ never meets $\overline{B}(0,\rho)$, then it also connects $x,y$ in the complement of $K$ because $K^*$ coincides with $K$ outside of $\overline{B}(0,\rho)$. If $\gamma$ meets $\overline{B}(0,\rho)$, then it also meets $\partial B(0,\rho)$ and it can only be at a point of $\partial B(0,\rho) \setminus Z$. By considering the first time at which $\gamma$ meets $\partial B(0,\rho)$, we see that $x$ is connected to $\partial B(0,\rho) \setminus Z$ in the complement of $K$. The same holds for $y$. By assumption, there exists a connected component $\mathcal{O}$ of $\Omega \setminus K$ which contains the domains $D_1$ and $D_2$, defined in (\ref{eq_domains}). The sets
    \begin{equation*}
        \set{x \in B(0,1) | x_N > 3\varepsilon_0} \quad \text{and} \quad \set{x \in B(0,1) | x_N < -3\varepsilon_0}
    \end{equation*}
    are connected subset of $\Omega \setminus K$ which meet $\mathcal{O}$ (because they contain $D_1$ and $D_2$) so they are also contained in $\mathcal{O}$. As a conclusion, we see that both $x$ and $y$ are connected to $\partial B(0,\rho) \setminus Z$ in the complement of $K$ and $\partial B(0,\rho) \setminus Z \subset \mathcal{O}$ so $x$ and $y$ are connected in the complement of $K$.

    The pair $(u^*,K^*)$ is a topological competitor of $(u,K)$ and $(u^*,K^*)$ in all balls $B(0,t)$ where $t \in (\rho,1)$ and we deduce
    \begin{equation*}
        \int_{B(0,\rho)} \CC{e(u)} \dd{x} + M^{-1} \HH^{N-1}(K \cap \overline{B}(0,\rho)) \leq \int_{B(0,\rho)} \CC{e(u^*)} \dd{x} + M \HH^{N-1}(Z) + h(1).
    \end{equation*}
    If $\varepsilon_2 \leq \varepsilonA$, where $\varepsilonA$ is the required parameter for the density lower bound (\ref{eq_AF}), then (\ref{eq_Jhyp}) yields $h(1) \leq \varepsilonA$ so
    \begin{equation*}
        M^{-1} \HH^{N-1}(K \cap B(0,\rho)) \geq C^{-1}.
    \end{equation*}
    On the other hand, $\HH^{N-1}(Z) \leq C \varepsilon_0$ and
    \begin{align*}
        \int_{B(0,\rho)} \CC{e(u^*)} \dd{x} &\leq C \int_{B(0,\rho)} \abs{\nabla v}^2 \dd{x} + C \int_{B(0,\rho)} \abs{e(w)}^2 \dd{x}\\
                                            &\leq C \varepsilon_0^{-2} \left(\omega_p(0,1) + J(0,1)^2\right)
    \end{align*}
    so the energy comparison yields $$C^{-1} \leq C \left(\varepsilon_0 + \varepsilon_0^{-2} \varepsilon_1 + \varepsilon_0^{-2} J(0,1)^2 \right) + \varepsilon_2,$$ where now $C \geq 1$ is a fixed constant which depends only on $N$, $\C$, $M$, $p$.
    We fix $\varepsilon_0$ and $\varepsilon_2$ small enough such that $C \varepsilon_0 \leq C^{-1}/6$, and $\varepsilon_2 \leq C^{-1}/6$. Then we choose $\varepsilon_1$ small enough such that $C \varepsilon_0^{-2} \varepsilon_1 \leq C^{-1}/6$. We arrive at $C/2 \leq C \varepsilon_0^{-2} J(0,1)^2$, which bounds $J(0,1)$ from below by constant which depends only on $N$, $\C$, $M$ and $p$. The statement follows for a suitable choice of $\tau_0$.
\end{proof}

\subsection{Size of holes through a projection.}

The following Lemma estimates the size of holes through a projection by a slicing technique. It reminds   an argument that 
{\sc Rigot} \cite{Rigot} performed   in the scalar case, but it is more intricate to use it here in the Griffith setting.
This complexity arises because the estimates involve only the component of the jump in the direction of the slicing.

\begin{lemma}\label{lem_slicing}
    Let $(u,K)$ be a pair in $\Omega$. Let $x_0 \in K$, $r_0 > 0$, and $\varepsilon \in (0,1/4)$ be such that $B(x_0,r_0) \subset \Omega$ and $\beta_K(x_0,r_0) \leq \varepsilon$. Let $P_0$, $\nu_0$ and $a_1$, $a_2$ be as in the beginning of Section \ref{section_jump}.
    Then for all unit vector $\nu \in \mathbf{S}^{N-1}$ such that $\abs{\nu - \nu_0} \leq \varepsilon$, we have
    \begin{equation*}
        J(\nu) \left(\frac{\HH^{N-1}\left(S(x_0,r_0,\nu,\varepsilon)\right)}{r_0^{N-1}}\right)^2 \leq C \varepsilon^{-1} \omega_1(x_0,r_0)^{1/2},
    \end{equation*}
    where $S(x_0,r_0,\nu,\varepsilon)$ is the size of holes through slicing in the direction $\nu$,
    \begin{equation*}
        S(x_0,r_0,\nu,\varepsilon) := P \cap B(x_0,(1 - 4\varepsilon) r_0) \setminus \pi_P(K \cap B(x_0,r_0)),
    \end{equation*}
    $P$ is the hyperplane $x_0 + \nu^\perp$, $\pi_P$ the orthogonal projection onto $P$, $J(\nu)$ is the component of the jump in the direction $\nu$,
    \begin{equation*}
        J(\nu) := \frac{\abs{(b_1 - b_2) \cdot \nu} + r_0 \abs{(A_1 - A_2) \nu}}{\sqrt{r_0}},
    \end{equation*}
    and $C \geq 1$ is a universal constant.
\end{lemma}
\begin{proof}
    The letter $C\geq 1$ denotes a universal constant   whose value might change from one line to another.
    Since all the quantities involved in the inequality are invariant under standard rescaling, see Remark \ref{rmk_scaling}, we can assume that $B(x_0,r_0) = B(0,1)$ without loss of generality.
    We let $\nu \in \mathbf{S}^{N-1}$ be a unit vector such that $\abs{\nu - \nu_0} \leq \varepsilon$.
    First of all, we observe that since
    \begin{equation*}
        K \cap B(0,1) \subset \set{\abs{x \cdot \nu_0} \leq \varepsilon},
    \end{equation*}
    and $\abs{\nu - \nu_0} \leq \varepsilon$, we also have
    \begin{equation}\label{eq_K0}
        K \cap B(0,1) \subset \set{\abs{x \cdot \nu } \leq 2 \varepsilon}.
    \end{equation}
    In what follows, we assume that $\nu$ is the last vector of the canonic basis to simplify the notations.
    We decompose each point $x \in \R^N$ as $x = x' + x_N e_N$, where $x' \in \R^{N-1}$ and $x_N = x \cdot e_N \in \R$.
    We let $a_1$ and $a_2$ denote the rigid motions that were defined in (\ref{eq_average_rigid}), at the beginning of Section \ref{section_jump}.

    \vspace{0.5cm}\noindent
    \emph{Step 1. Building an auxiliary function.}
    We build a function $v \in W^{1,2}_{\mathrm{loc}}(B(0,1) \setminus K;\R^N)$ such that
    \begin{align*}
        v(x) = a_1(x) &\ \text{in} \ B(0,1) \cap \set{x_N \geq 4 \varepsilon}\\
        v(x) = a_2(x) &\ \text{in} \ B(0,1) \cap \set{x_N \leq -4 \varepsilon},
    \end{align*}
    and the following estimate holds
            \begin{equation*}
                \int_{B(0,1) \setminus K} \abs{e(v)} \dd{x} \leq C \varepsilon^{-1} \int_{B(0,1) \setminus K} \abs{e(u)} \dd{x}.
            \end{equation*}
    We consider a smooth cut-off function $\varphi_1 : \R^N \to \R$ equal to $1$ on $\set{x_N \geq 4 \varepsilon}$, equal to $0$ on $\set{x_N \leq 2\varepsilon}$, with $0 \leq \varphi_1 \leq 1$ and $\abs{\nabla \varphi_1} \leq C \varepsilon^{-1}$. We also consider an other cut-off function $\varphi_2 : \R^N \to \R$ equal to $1$ on $\set{x_N \leq -4 \varepsilon}$, equal to $0$ on $\set{x_N \geq -2\varepsilon}$, with $0 \leq \varphi_2 \leq 1$ and $\abs{\nabla \varphi_2} \leq C \varepsilon^{-1}$.
    We finally define
    \begin{equation*}
        v(x) := \varphi_1(x) a_1(x) + \varphi_2(x) a_2(x) + (1 - \varphi_1(x)) (1 - \varphi_2(x)) u(x).
    \end{equation*}
    This function $v$ defined above clearly belongs to $W^{1,2}_{\mathrm{loc}}(B(0,1) \setminus K;\R^N)$ and satisfies properties (1) and (2) of the statement.
    Let us estimate the energy of $v$ in the region $B(0,1) \cap \set{x_N > 2\varepsilon} \setminus K$.
    In this domain, we know that $\varphi_2 = 0$ so that the expression of $v$ reduces to $v = \varphi_1 a_1 + (1 - \varphi_1) u$ and therefore 
    \begin{align*}
        \abs{e(v)} &\leq \abs{\nabla \varphi_1} \abs{a_1 - u} + (1 - \varphi_1) \abs{e(u)}\\
                   &\leq C \varepsilon^{-1} \abs{a_1 - u} + \abs{e(u)}.
    \end{align*}
    We recall that the rigid motion $a_1$ is the average rigid motion of $u$ in the domain $D_1 := B((3/4) \nu_0,1/8)$.
    Since $\abs{e_N - \nu_0} \leq \varepsilon$ and $\varepsilon \leq 1/4$, we have $e_N \cdot \nu_0 \geq 1 - \varepsilon^2/2 > 5/6$ and thus for $x \in D_1$,
    \begin{equation}\label{eq_D0}
        \abs{x_N} \geq (3/4) (e_N \cdot \nu_0) - 1/8 > (3/4) \cdot (5/6) - 1/8 = 1/2.
    \end{equation}
    Hence, $D_1$ is contained in $B(0,1) \cap \set{x_N > 2\varepsilon}$, which is a Lipschitz domain disjoint from $K$ by (\ref{eq_K0}). So by Korn-Poincaré inequality, we have
    \begin{equation*}
        \int_{B(0,1) \cap \set{x_N > 2\varepsilon}} \abs{u - a_1} \dd{x} \leq C \int_{B(0,1)} \abs{e(u)} \dd{x}.
    \end{equation*}
    We conclude that
    \begin{equation*}
        \int_{B(0,1) \cap \set{x_N > 2\varepsilon}} \abs{e(v)} \dd{x} \leq C \varepsilon^{-1} \int_{B(0,1)} \abs{e(u)} \dd{x}.
    \end{equation*}
    We can estimate the energy of $v$ in $B(0,1) \cap \set{x_N < -2\varepsilon}$ in the same way.
    And in the domain $B(0,1) \cap \set{- 2\varepsilon < x_N < 2\varepsilon} \setminus K$, we have $v = u$ so there is nothing to do.

    \vspace{0.5cm}\noindent
    \emph{Step 2. Controlling the size of holes in the projection by slicing.} This step is based on the elementary observation that 
    \begin{equation*}
        \frac{\dd}{\dd{t}} \left[v(x + t e_{N}) \cdot e_N\right] = \left(e(v)(x + te_N) e_N\right) \cdot e_N.
    \end{equation*}
    Let $S := P \cap B(0,1-4\varepsilon) \setminus \pi_P(K \cap B(0,1))$, where we recall that $P = \set{x_N = 0}$.
    Then for all $x' \in S$, we can integrate along a vertical segment from $x^- = x' - 4\varepsilon e_N$ to $x^+ = x' +  4\varepsilon e_N$.
    This yields
    \begin{equation*}
        \left(a_2(x^+) - a_1(x^-)\right) \cdot e_N = \int_{-4\varepsilon}^{4\varepsilon} \left(e(v)(x' + t e_N) e_N\right) \cdot e_N \dd{t}.
    \end{equation*}
    Then we apply Fubini and integrate with respect to $x' \in S$, namely,
    \begin{align*}
        \int_S \abs{\left(a_2(x^+) - a_1(x^-)\right) \cdot e_N} \dd{x'} &= \int_S \int_{-4 \varepsilon}^{4\varepsilon} \left(e(v)(x' + t e_N) e_N\right) \cdot e_N \dd{t} \dd{x'}\\
                                                                        &\leq \int_{B(0,1)} \abs{e(v)} \dd{x}\\
                                                                        &\leq C \varepsilon^{-1} \int_{B(0,1)} \abs{e(u)} \dd{x}.
    \end{align*}
    Now we recall that $a_i(x) = b_i + A_i x$, where $A$ is a skew-symmetric matrix so
    \begin{align*}
        a_i(x' \pm 2 \varepsilon e_N) \cdot e_N &= b_i \cdot e_N + (A_i x') \cdot e_N \pm 2 \varepsilon (A_i e_N) \cdot e_N \\
                                                &= b_i \cdot e_N - (A_i e_N) \cdot x'
    \end{align*}
    and we arrive at
    \begin{equation*}
        \int_S \abs{(b_2 - b_1) \cdot e_N - x' \cdot \left((A_2 - A_1) e_N\right)} \dd{x'} \leq C \varepsilon^{-1} \int_{B(0,1)} \abs{e(u)} \dd{x}.
    \end{equation*}
In view of Lemma \ref{lem_int_rigid} in Appendix (applied in $\R^{N-1}$), this gives finally
    \begin{equation*}
        \HH^{N-1}(S)^2 \left(\abs{(b_2 - b_1) \cdot e_N} + \abs{(A_2 - A_1) e_N}\right) \leq C \varepsilon^{-1} \int_{B(0,1)} \abs{e(u)} \dd{x}.
    \end{equation*}
\end{proof}

\subsection{Proof of Proposition \texorpdfstring{\ref{prop_mainHole}}{4.1}}

\newcommand{\varepsilonW}{\varepsilon_2}
\newcommand{\varepsilonB}{\varepsilon_1}
\newcommand{\varepsilonG}{\varepsilon_0}

\begin{proof}[Proof of Proposition \ref{prop_mainHole}]
    As usual, we let $C\geq 1$ denote a generic constant   which depends only on $N$, $\C$, $M$ and $p$.
    We let $P_0$, $\nu_0$, $D_1$, $D_2$ and $a_1$, $a_2$ be as in the beginning of Section \ref{section_jump}.
    We fix $\varepsilon > 0$ and we let $\varepsilonG, \varepsilonB, \varepsilonW \in (0,1/4)$ be small parameters (they will be chosen small enough depending on $\varepsilon$) such that
    \begin{equation*}
        \beta(x_0,r_0) \leq \varepsilonB,\quad \omega_p(x_0,r_0) \leq \varepsilonW,\quad h(r_0) \leq \varepsilonG.
    \end{equation*}
    We consider a unit vector $\nu \in \mathbf{S}^{N-1}$ such that $\abs{\nu - \nu_0} \leq \varepsilonB$ and we let $P = x_0 + \nu^\perp$ and $\pi_P$ denote the orthogonal projection onto $P$.
    Since orthogonal projections are $1$-Lipschitz, we can bound the measure of $K$ from below by the measure of its projection
    \begin{align*}
        \HH^{N-1}(K \cap B(x_0,r_0)) &\geq \HH^{N-1}(\pi_P(K \cap B(x_0,r_0))\\
                                     &\geq \omega_{N-1} r_0^{N-1} - \HH^{N-1}(P \cap B(x_0,r_0) \setminus \pi_P(K \cap B(x_0,r_0)).
    \end{align*}
    Now the goal of the proof is to control
    \begin{equation*}
        \HH^{N-1}(P \cap B(x_0,r_0) \setminus \pi_P(K \cap B(x_0,r_0))
    \end{equation*}
    for a suitable choice of vector $\nu$. We can first bound
    \begin{equation*}
        \HH^{N-1}(P \cap B(x_0,r_0) \setminus B(x_0, (1 - 4 \varepsilonB)r_0)) \leq C \varepsilonB r_0^{N-1},
    \end{equation*}
    and are left to deal with
    \begin{equation*}
        \HH^{N-1}(P \cap B(x_0,(1 - 4\varepsilonB) r_0) \setminus \pi_P(K \cap B(x_0,r_0))).
    \end{equation*}

    We focus first on the case where $K$ separates the domain $D_1$ and $D_2$.
    We recall the fact seen just below (\ref{eq_D0}), that $D_1$ is contained in $B(0,1) \cap \set{x \cdot \nu > 2\varepsilonB}$ and $D_2$ is contained in $B(0,1) \cap \set{x \cdot \nu < -2\varepsilonB}$, which are convex domains disjoint from $K$.
    Hence, for all $x \in P \cap B(x_0,(1 - 4\varepsilonB)r_0)$, the segment $x + [-2\varepsilonB,2\varepsilonB] \nu$ must meet $K$ otherwise it could be used to connect $D_1$ and $D_2$. Thus the projection $\pi_P(K \cap B(x_0,r_0))$ contains $P \cap B(x_0,(1-4\varepsilonB)r_0)$. We conclude in this case that
    \begin{equation*}
        \HH^{N-1}(K \cap B(x_0,r_0)) \geq \HH^{N-1}(\pi_P(K \cap B(x_0,r_0)) \geq \omega_{N-1} r_0^{N-1} - C \varepsilonB r_0^{N-1}.
    \end{equation*}
    It then suffices to choose $\varepsilonB$ small enough (depending on $N$, $\varepsilon$) to conclude the theorem statement.

    We now assume that $K$ does not separate the domain $D_1$ and $D_2$ and this will allow us to use Lemma \ref{lem_Jinit}.
    We know by Lemma \ref{lem_slicing} that for all $\nu \in \mathbf{S}^{N-1}$ such that $\abs{\nu - \nu_0} \leq \varepsilonB$, we have
    \begin{equation*}
        J(\nu) \left(\frac{\HH^{N-1}\left(S(x_0,r_0,\nu,\varepsilon)\right)}{r_0^{N-1}}\right)^2 \leq C \varepsilonB^{-1} \omega_1(x_0,r_0)^{1/2},
    \end{equation*}
    where
    \begin{equation}
        S(x_0,r_0,\nu,\varepsilon) = P \cap B(x_0,(1 - 4\varepsilonB) r_0) \setminus \pi_P(K \cap B(x_0,r_0)).
    \end{equation}
    $P = x_0 + \nu^\perp$ and $\pi_P$ is the orthogonal projection onto $P$.
    We are then looking for a vector $\nu$ close to $\nu_0$ such that $J(\nu)$ is bounded from below.
    To simplify the notations, we set $b := b_2 - b_1$ and $A := A_2 - A_1$.
    According to Lemma \ref{lem_int_rigid2}), we have
    \begin{equation*}
        \int_{\nu \in \mathbf{S}^{N-1} \cap B(\nu_0,\varepsilonB)} \abs{b \cdot \nu} + \abs{A \nu} \dd{\HH^{N-1}(\nu)} \geq C(\varepsilonB)^{-1} \left(\abs{b} + r_0 |A|\right)
    \end{equation*}
    where $C(\varepsilonB) \geq 1$ also depends on $\varepsilonB$ and is allowed to take a bigger value in the next lines.
    We can thus find a vector $\nu \in \mathbf{S}^{N-1}$ such that $\abs{\nu - \nu_0} \leq \varepsilonB$ and
    \begin{equation*}
        J(\nu) \geq C(\varepsilonB)^{-1} J(x_0,r_0),
    \end{equation*}
    where $J(x_0,r_0)$ is the normalized jump defined in (\ref{eq_normalized_jump}).
    For this choice of $\nu$, we have
    \begin{equation*}
        J(x_0,r_0) \left(\frac{\HH^{N-1}\left(S(x_0,r_0,\nu,\varepsilon)\right)}{r_0^{N-1}}\right)^2 \leq C(\varepsilonB) \omega_1(x_0,r_0)^{1/2}.
    \end{equation*}
    Now, we let $\tau_0$ be the constant of Lemma \ref{lem_Jinit}, which depends only on $N$, $\C$, $M$, $p$, and we take $\varepsilonB \leq \tau_0/3$, $\varepsilonW \leq \tau_0/3$ and $\varepsilonG = \tau_0/3$ so that $J(x_0,r_0) \geq \tau_0^{-1}$. Using also the fact that $\omega_1 \leq \omega_p$, we arrive at
    \begin{equation*}
        \left(\frac{\HH^{N-1}\left(S(x_0,r_0,\nu,\varepsilon)\right)}{r_0^{N-1}}\right)^2 \leq C(\varepsilonB) \varepsilonW^{1/2}.
    \end{equation*}
    We conclude that
    \begin{equation*}
        \HH^{N-1}(K \cap B(x_0,r_0)) \geq \left(\omega_{N-1} - C \varepsilonB - C(\varepsilonB) \varepsilonW^{1/4}\right) r_0^{N-1}.
    \end{equation*}
    We can first fix $\varepsilonB$ such that $\varepsilonB \leq \omega_{N-1} \varepsilon/2$ and then $\varepsilonW$ even smaller such that $C(\varepsilonB) \varepsilonW^{1/4} \leq \omega_{N-1} \varepsilon/2$, which yields
    \begin{equation*}
        \HH^{N-1}(K \cap B(x_0,r_0)) \geq \omega_{N-1} (1 - \varepsilon) r_0^{N-1},
    \end{equation*}
    and finishes the proof.
\end{proof}

\section{Uniform concentration property}

In this section we will prove the uniform concentration property that was announced in the introduction, i.e. in Theorem \ref{prop_uc}. We recall the definition of a uniformly concentrated sequence given in \cite[Section 35]{DavidBOOK}.

\begin{definition}\label{defi_uc}
    Let $(E_i)_i$ and $E$ be relatively closed subsets of $\Omega$. We say that the sequence $(E_i)_i$ is uniformly concentrated with respect to $E$ provided that for all $\varepsilon \in (0,1)$, there exists a constant $C(\varepsilon) \geq 1$ such that the following holds. For all $x \in E$, there exists $r(x) > 0$ such that for all $0 < r \leq r(x)$, for all $i$ large enough, we can find a ball $B(y_i,\rho_i) \subset \Omega \cap B(x,r)$ with $\rho_i \geq C(\varepsilon)^{-1} r$ and
    \begin{equation*}
        \HH^{N-1}(E_i \cap B(y_i,\rho_i)) \geq (1 - \varepsilon) \omega_{N-1} \rho_i^{N-1},
    \end{equation*}
    where $\omega_{N-1}$ is the measure of the $(N-1)$-dimensional unit disk.
\end{definition}
As mentioned in introduction, this property implies the lower semi-continuity of the area,
\begin{equation}\label{eq_uc}
    \HH^{N-1}(E) \leq \liminf_{i \to +\infty} \HH^{N-1}(E_i).
\end{equation}
We refer to \cite{MSBOOK} or \cite[Theorem 35.4]{DavidBOOK} for a proof.
We then show that for a Griffith almost-minimizers, the density of $K$ is greater than $1 - \varepsilon$ when the normalized elastic energy is small.
This result improves Proposition \ref{prop_mainHole} by removing the flatness assumption and finds its intuition in the fact that $K$ behaves like a minimal sets in regime of low elastic energy.

\begin{proposition}\label{prop_omega}
    For each $p \in (2(N-1)/N,2]$ and $\varepsilon > 0$, there exists a constant $\varepsilon_0 > 0$ (depending on $N$, $\C$, $p$, $\varepsilon$) such that the following holds.
    Let $(u,K)$ be a topological almost-minimizer with gauge $h$ in $\Omega$. For all $x_0 \in K$ and for all $r_0 > 0$ such that $B(x_0,r_0) \subset \Omega$ and 
    \begin{equation}\label{eq_omegaHyp}
        \omega_p(x_0,r_0) + h(r_0) \leq \varepsilon_0,
    \end{equation}
    we have
    \begin{equation*}
        \HH^{N-1}(K \cap B(x_0,r_0)) \geq (1 - \varepsilon) \omega_{N-1} r_0^{N-1},
    \end{equation*}
    where $\omega_{N-1}$ is the measure of the unit $(N-1)$-dimensional disk.
\end{proposition}


\begin{proof}[Proof of Proposition \ref{prop_omega}]
    As usual, we let $C\geq 1$ denote a generic constant  which depends only $N$, $\C$, $p$.
    By standard rescaling, we assume that $B(x_0,r_0) = B(0,1)$ without loss of generality.

    \vspace{0.5cm}\noindent
    \emph{Step 1. Contradiction and compactness.}
    We proceed by contradiction and find a parameter $c \in (0,1)$ and sequence of topological almost minimizers $(u_i,K_i)$ in $B(0,1)$ such that $0 \in K_i$,
    \begin{equation}\label{eq_omega_init}
        \int_{B(0,1)} \abs{e(u_i)}^p \dd{x} + h_i(1) \to 0 
    \end{equation}
    and
    \begin{equation*}
        \HH^{N-1}(K_i \cap B(0,1)) < (1 - c) \omega_{N-1}.
    \end{equation*}
    Since $h_i(1) \to 0$, we can extract a subsequence (not relabelled) such that for all $i$, $h_i(1) \leq \varepsilonA$, where $\varepsilonA$ is the constant needed for (\ref{eq_AF}) and (\ref{eq_AF2}). We thus have
    \begin{equation*}
        \sup_i \int_{B(0,1)} \abs{e(u)}^2 \dd{x} + \HH^{N-1}(K_i \cap B(0,1)) < +\infty,
    \end{equation*}
    and we can extract a subsequence such that the measures $(\HH^{N-1} \mres K_i)_i$ converge to a measure $\mu$ and such that the pairs $(u_i,K_i)_i$ converge to a pair $(u,K)$ in $B(0,1)$.
    Since $h_i(1) \leq \varepsilonA$ uniformly, there exists a constant $C \geq 1$ (depending only on $N$, $\C$) such that for all $i$, for all $x \in K_i$ and for all $r > 0$ such that $B(x_i,r) \subset B(0,1)$, we have
    \begin{equation*}
        C^{-1} r^{N-1} \leq \HH^{N-1}(K_i \cap B(x,r)) \leq C r^{N-1}.
    \end{equation*}
    It follows that for all $x \in K$ and for all $r > 0$ such that $B(x,r) \subset B(0,1)$, we have
    \begin{equation*}
        C^{-1} r^{N-1} \leq \mu(B(x,r)) \leq C r^{N-1},
    \end{equation*}
    and
    \begin{equation*}
        C^{-1} r^{N-1} \HH^{N-1}(K \cap B(x,r)) \leq C r^{N-1},
    \end{equation*}
    see Lemma \ref{lem_uniformAF}.
    Moreover, by application of Proposition \ref{prop_weak_limit} and the fact that $$\lim_{i \to +\infty} \int_{B(0,1)} \abs{e(u_i)} \dd{x} + h_i(1) = 0,$$ we know that for all $x \in K$, for all $r > 0$ such that $\overline{B}(x,r) \subset B(0,1)$ and for all topological competitor $L$ of $K$ in $B = B(x,r)$, we have
    \begin{equation}\label{eq_mu_minimality}
        \mu(B(x,r)) \leq \HH^{N-1}(L \cap \overline{B}(x,r)).
    \end{equation}

    \vspace{0.5cm}\noindent
    \emph{Step 2. The limit $K$ is rectifiable}.
    As A limit of rectifiable sets may not be rectifiable in general, we have no other choice than to take advantage of the minimality property (\ref{eq_mu_minimality}) satisfied by the limit.
    For this purpose, we borrow a Federer-Fleming type argument from \cite{Lab}.
    Since $K$ is a Borel set with finite measure in $B(0,1)$, it can be decomposed as disjoint union $K = E \cup F$ of two Borel sets with $E$ being rectifiable and $F$ being purely unrectifiable. We proceed by contradiction and assume that $\HH^{N-1}(F) > 0$.
    We know by standard density theorems \cite[Theorem 6.2(2)]{Mattila} that for $\HH^{N-1}$-a.e. $x \in F$, we have
    \begin{equation*}
        \lim_{r \to 0} r^{1-N} \HH^{N-1}(E \cap B(x,r)) = 0.
    \end{equation*}
    Let us fix such a point $x \in K$ and let us consider a radius $r > 0$ such that $\overline{B}(x,r) \subset \Omega$.
    For $z \in B(x,r) \setminus K$, we let $\phi_z : \overline{B}(x,r) \setminus \set{z} \to \R^N$ be the radial projection centered at $z$ onto $\partial B(x,r)$.
    We are going to see that for a suitable choice of center $z$, the radial projection cancels the purely unrectifiable part of $K$.
    For $d = 1,\ldots,N-1$, we let $G(N,d)$ denote the Grassmannian manifold of all $d$-dimensional linear subspace of $\R^N$.
    We let $\gamma_{N,d}$ denote the canonic probability measure on $G(N,d)$ and we refer to \cite[Chapter 3, \S 3.9]{Mattila} for the definition.
    We shall know that for all Borel set $\mathcal{S} \subset G(N,d)$,
    \begin{equation*}
        \gamma_{N,d}(\mathcal{S}) = \gamma_{N,N-d}(\set{V^\perp | V \in \mathcal{S}}),
    \end{equation*}
    and that for all Borel set $S \subset \partial B(0,1)$, we have
    \begin{equation}\label{eq_SG}
        \HH^{N-1}(S) \leq C \gamma_{N,1}\left(\set{L | L \cap S \ne \emptyset}\right).
    \end{equation}
    Let us justify that for all shifted center $x_0 \in B(0,1/2)$ and for all Borel set $S \subset \partial B(0,1)$, we also have
    \begin{equation}\label{eq_SG2}
        \HH^{N-1}(S) \leq C \gamma_{N,1}\left(\set{L | (x_0 + L) \cap S \ne \emptyset}\right).
    \end{equation}
    If we let $f$ denote the radial projection centered at $x_0$ onto $\partial B(x_0,2)$, then the restriction of $f$ on $S \subset \partial B(0,1)$ is $C$-biLipschitz so $\HH^{N-1}(S) \leq C \HH^{N-1}(f(S))$. 
    Using a rescaled version of (\ref{eq_SG}) in $B(x_0,2)$, we can estimate
    \begin{equation*}
        \HH^{N-1}(f(S)) \leq C \gamma_{N,1}\left(\set{L | (x_0 + L) \cap f(S) \ne \emptyset}\right)
    \end{equation*}
    and we observe from the definition of $f$ that $(x_0 + L) \cap \phi(S) \ne \emptyset$ if and only if $(x_0 + L) \cap S \ne \emptyset$. This proves (\ref{eq_SG2}).
    Let us come back to the ball $B(x,r)$ and the radial projection $\phi_z$ centered on $z$ onto $\partial B(x,r)$.
    Using a rescaled version of (\ref{eq_SG2}) in $B(x,r)$, we deduce that for all $z \in B(x,r/2) \setminus K$, we have
    \begin{equation*}
        \HH^{N-1}(\phi_z(K \cap \overline{B}(x,r))) \leq C r^{N-1} \gamma_{N,1}\left(\set{L | \left(L + z\right) \cap K \cap \overline{B}(x,r) \ne \emptyset}\right).
    \end{equation*}
    With the help of Fubini, we can estimate that on average
    \begin{align*}
        \begin{split}
        &\fint_{B(x,r/2) \setminus K} \HH^{N-1}(\phi_z(K \cap \overline{B}(x,r))) \dd{\HH^{N-1}}(z)\\
        &\qquad \leq C r^{-1} \int_{B(x,r/2) \setminus K} \int_{G(N,1)} \gamma_{N,1}\left(\set{L | (x_0 + L) \cap K \cap \overline{B}(x,r) \ne \emptyset}\right) \dd{L} \dd{\HH^{N-1}}(z)
        \end{split}\\
                                                                        &\qquad \leq C r^{-1} \int_{G(N,1)} \abs{\set{z \in B(x,r/2) | (x_0 + L) \cap K \cap \overline{B}(x,r) \ne \emptyset}} \dd{L}\\
                                                                        &\qquad \leq C r^{-1} \int_{G(N,N-1)} \abs{\set{z \in B(x,r/2) | (x_0 + V^\perp) \cap K \cap \overline{B}(x,r) \ne \emptyset}} \dd{V}\\
                                                                        &\qquad \leq C \int_{G(N,N-1)} \HH^{N-1}(p_V(K \cap \overline{B}(x,r)) \dd{V}.
    \end{align*}
    According to the Besicovitch-Federer projection Theorem \cite[Theorem 18.1]{Mattila}, we have for almost-all hyperplanes $V \in G(N,N-1)$,
    \begin{equation*}
        \HH^{N-1}(p_V(F \cap \overline{B}(x,r)) = 0
    \end{equation*}
    and from the fact that orthogonal projections are $1$-Lipschitz, we have for all $V \in G(N,N-1)$, 
    \begin{equation*}
        \HH^{N-1}(p_V(E \cap \overline{B}(x,r)) \leq \HH^{N-1}(E \cap \overline{B}(x,r)),
    \end{equation*}
    whence
    \begin{equation*}
        \fint_{B(x,r/2) \setminus K} \HH^{N-1}(\phi_z(K \cap \overline{B}(x,r))) \dd{\HH^{N-1}}(z) \leq C \HH^{N-1}(E \cap \overline{B}(x,r)).
    \end{equation*}
    This allows to find a center $z \in B(x,r) \setminus K$ such that
    \begin{equation*}
        \HH^{N-1}(\phi_z(K \cap \overline{B}(x,r))) \leq C \HH^{N-1}(E \cap \overline{B}(x,r)).
    \end{equation*}
    We extend $\phi_z$ out of $\overline{B}(x,r)$ by setting $\phi_z = \mathrm{id}$ and thus $L = \phi_z(K)$ is a topological competitor of $K$ in all balls $B(x,t)$ where $t > r$ and $\overline{B}(x,t) \subset \Omega$, see Remark \ref{rmk_borsuk}.
    We apply (\ref{eq_mu_minimality}) and we use $\mu(B(x,r)) \geq C^{-1} r^{N-1}$ to obtain
    \begin{equation*}
        C^{-1} r^{N-1} \leq \HH^{N-1}(E \cap \overline{B}(x,r)).
    \end{equation*}
    Remember that $x$ is a point such that $\lim_{r \to 0} r^{1-N} \HH^{N-1}(E \cap B(x,r)) = 0$ so we find a contradiction if $r$ is small enough. This proves that we actually have $\HH^{N-1}(F) = 0$ and thus $K$ is rectifiable.
    Note that as a standard consequence of rectifiability and Ahlfors-regularity, we have for $\HH^{N-1}$-a.e. $x \in K$,
    \begin{equation*}
        \lim_{r \to 0} \beta_K(x,r) = 0.
    \end{equation*}


    \vspace{0.5cm}\noindent
    \emph{Step 3. Lower semi-continuity of the area.}
    Our goal now is to prove that for $\HH^{N-1}$-a.e. $x \in K$, we have
    \begin{equation}\label{eq_mu}
        \limsup_{r \to 0} \frac{\mu(B(x,r))}{\omega_{N-1} r^{N-1}} \geq 1.
    \end{equation}
    It will follow from standard density theorems
    \cite[Theorem 6.9(2)]{Mattila} that $\mu \geq \HH^{N-1} \mres K$.
    Let us fix $x \in K$ such that $\lim_{r \to 0} \beta_K(x,r) = 0$.
    By convergence in Hausdorff distance, there exists a sequence of points $x_i \in K_i$ such that $x_i \to x$.
    For $\varepsilon > 0$, we let $\varepsilon_0, \varepsilon_1 \in (0,1)$ be the associated constant given by Proposition \ref{prop_mainHole}. There exists a small radius $r(x) > 0$ such that $\overline{B}(x,r(x)) \subset B(0,1)$ and for all $0 < r \leq r(x)$, it holds $\beta_K(x,2r) < \varepsilon_1/8$.
    The radius $r$ being fixed, let us check that for $i$ big enough, we have $\beta_{K_i}(x_i,r) \leq \varepsilon_1/2$.
    There exists an hyperplane $P$ passing through $x$ such that
    \begin{equation*}
        K \cap B(x,2r) \subset \set{y \; :\; \mathrm{dist}(y,P) < \varepsilon r/4}.
    \end{equation*}
    As $x_i \to x$ and $K_i \to K$, more precisely using (\ref{eq_KCV2}), we see that for $i$ big enough
    \begin{equation*}
        K_i \cap B(x_i,r) \subset K_i \cap \overline{B}(x,3r/2) \subset \set{\mathrm{dist}(\cdot,P) < \varepsilon_1 r/4}.
    \end{equation*}
    Let $P_i$ be the hyperplane parallel to $P$ passing through $x_i$.
    For $i$ big enough, we have $\abs{x_i - x} \leq \varepsilon_1 r/4$ and in particular $P_i$ is a distance $\leq \varepsilon_1 r/4$ from $P$ so
    \begin{equation*}
        K_i \cap B(x_i,r) \subset \set{\mathrm{dist}(\cdot,P_i) < \varepsilon_1 r/2}
    \end{equation*}
    and this justifies that $\beta_{K_i}(x_i,r) \leq \varepsilon_1/2$.
    For $i$ big enough, we also have
    \begin{equation*}
        r^{1-2N/p} \left(\int_{B(x_i,r)} \abs{e(u_i)}^p \dd{x}\right)^{2/p} \leq \varepsilon_1/2
    \end{equation*}
    and $h_i(r) \leq \varepsilon_0$ because of the initial assumption (\ref{eq_omega_init}).
    We can then apply Proposition \ref{prop_mainHole} in $B(x_i,r)$ for $i$ big enough, which shows that
    \begin{equation*}
        \HH^{N-1}(K_i \cap B(x_i,r)) \geq \omega_{N-1} (1 - \varepsilon) r^{N-1}.
    \end{equation*}
    Passing to the limit, we arrive at
    \begin{equation*}
        \mu(\overline{B}(x,r)) \geq \omega_{N-1}(1 - \varepsilon) r^{N-1}.
    \end{equation*}
    From the fact that this holds for $\mu(\overline{B}(x,r))$ for all $0 < r \leq r(x)$, one can also deduce that this holds for $\mu(B(x,r))$ for all $0 < r \leq r(x)$.
    We conclude
    \begin{equation*}
        \liminf_{r \to 0} \frac{\mu(B(x,r))}{\omega_{N-1} r^{N-1}} \geq 1 - \varepsilon,
    \end{equation*}
    and we finally let $\varepsilon \to 0$ to prove our claim (\ref{eq_mu}).

    \vspace{0.5cm}\noindent
    \emph{Step 4. Conclusion.}
    Given that $\mu \geq \HH^{N-1} \mres K$, the minimality condition (\ref{eq_mu_minimality}) actually yields that for all $x \in K$, for all $r > 0$ such that $\overline{B}(x,r) \subset B(0,1)$ and for all topological competitor $L$ of $K$ in $B(x,r)$, we have
    \begin{equation*}
        \HH^{N-1}(K \cap B(x,r)) \leq \HH^{N-1}(L \cap \overline{B}(x,r)).
    \end{equation*}
    In fact, it is possible to remove the closure of the ball at the right-hand side. Here are more details. For small $t > r$ such that $\overline{B}(x,t) \subset B(0,1)$, the set $L$ is a topological competitor of $K$ in $B(x,t)$ so we can replace $B(x,r)$ by $B(x,t)$ in the above inequality. Then we let $t \to r^+$ to obtain
    $\HH^{N-1}(K \cap \overline{B}(x,r)) \leq \HH^{N-1}(L \cap \overline{B}(x,r))$
    and we use the fact that $K$ coincides with $L$ on $\partial B(x,r)$ to recover
    \begin{equation*}
        \HH^{N-1}(K \cap B(x,r)) \leq \HH^{N-1}(L \cap B(x,r)).
    \end{equation*}
    This means that $K$ is a minimal set, see Definition \ref{defi_minimalset}, and in particular, it has a minimal area under continuous deformation, see Remark \ref{rmk_borsuk}.
    These sets have monotone densities (\cite[Proposition 5.16]{David2009} or \cite[Chapter 3 \S 17]{Simon}), i.e., for all $x \in K$ and for all $0 < r < R$ such that $B(x,R) \subset B(0,1)$, we have
    \begin{equation*}
        r^{1-N} \HH^{N-1}(K \cap B(x,r)) \leq R^{1-N} \HH^{N-1}(K \cap B(x,R)).
    \end{equation*}
    As a consequence, the limit
    \begin{equation*}
        \theta(x) := \lim_{r \to 0} r^{1-N} \HH^{N-1}(K \cap B(x,r))
    \end{equation*}
    exists and is finite at all points $x \in E$.
    We also know by rectifiability that for $\HH^{N-1}$-a.e. $x \in K$, we have $\theta(x) = \omega_{N-1}$.
    As $K$ is coral and contains $0$, we have $\HH^{N-1}(K \cap B(0,\varepsilon)) > 0$ for all $\varepsilon \in (0,1)$ and therefore we can find $x \in E \cap B(0,\varepsilon))$ such that $\theta(x) = \omega_{n-1}$. Then, we estimate by monotonicity
    \begin{align*}
        \omega_{N-1} &\leq (1 - \varepsilon)^{1-N} \HH^{N-1}(K \cap B(x,1-\varepsilon))\\
                     &\leq (1 - \varepsilon)^{1-N}\HH^{N-1}(K \cap B(0,1))
    \end{align*}
    and letting $\varepsilon \to 0$ yields
    \begin{equation*}
        \omega_{N-1} \leq \HH^{N-1}(K \cap B(0,1)).
    \end{equation*}
    This contradicts the fact that $\mu \geq \HH^{N-1} \mres K$ and the initial assumption that $\HH^{N-1}(K_i \cap B(0,1)) \leq \omega_{N-1}(1 - c)$.
\end{proof}

We are now ready to prove the concentration property stated in Theorem \ref{prop_uc}. Notice that in the Mumford-Shah setting, uniform concentration is also known to hold for quasiminimizers.
We expect that this should also be the case in the Griffith setting but our approach, which relies on Proposition \ref{prop_omega}, is not suitable for quasiminimizers.

\begin{proof}[Proof of Theorem \ref{prop_uc}]
    The letter $C\geq 1$ denotes a generic constant   which depends only on $N$ and $\C$.
    Let us fix an exponent $p \in (2(N-1)/N,2]$, say the middle point in the interval so that it depends only on $N$.
    Let us fix $\varepsilon > 0$ and let $\varepsilon_0$ be the associated constant given by Proposition \ref{prop_omega}.
    We want to find a smaller shifted ball where Proposition \ref{prop_omega} applies and for this, we recall the Carleson estimate proved in \cite{DavidBOOK}. It says that for all $x \in K$, for all $r > 0$ such that $B(x,2r) \subset \Omega$ and $h(2r) \leq \varepsilonA$, we have
    \begin{equation*}
        \int_{y \in K \cap B(x,r)} \! \int_0^r \! \omega_p(y,t) \frac{\dd{t}}{t} \dd{\HH^1(y)} \leq C r^{N-1}.
    \end{equation*}
    The proof in \cite{DavidBOOK} is performed with $\abs{\nabla u}^2$ but readily applies with $\abs{e(u)}^2$ since it only uses the Ahlfors-regularity properties (\ref{eq_AF}), (\ref{eq_AF2}).
    We are going to use this inequality to find a constant $C_0 \geq 1$ (depending on $N$, $\C$, $\varepsilon_0$) such that the following holds: for all $x \in K$, for all $r > 0$ such that $B(x,r) \subset \Omega$ and $h(r) \leq \varepsilonA$, there exists $y \in K \cap B(x,r/2)$ and $t \in (C_0^{-1} r,r/2)$ such that $\omega_p(y,t) \leq \varepsilon_0/2$.
    Indeed, if this is not the case for a given constant $C_0 \geq 1$, then
    \begin{align*}
        C r^{N-1} \geq \int_{y \in K \cap B(x,r/2)} \! \int_0^{r/2} \! \omega_p(y,t) \frac{\dd{t}}{t} \dd{\HH^1(y)} &\geq \int_{y \in K \cap B(x,r/2)} \! \int_{C_0^{-1} r}^{r/2} \! \omega_p(y,t) \frac{\dd{t}}{t} \dd{\HH^1(y)}\\
                                                                                                                    &\geq \frac{\varepsilon_0}{2} \HH^{N-1}(K \cap B(x,r/2)) \ln\left(\frac{C_0}{2}\right)\\
                                                                                                                    &\geq C^{-1} \varepsilon_0 \ln(\frac{C_0}{2}) r^{N-1}.
    \end{align*}
    We reach a contradiction if $C_0$ is too big (depending on $N$, $\C$, $\varepsilon_0$).
    We conclude that for all $x \in K$, for all $r > 0$ such that $B(x,r) \subset \Omega$ and $h(r) \leq \varepsilonA$, there exists $y \in K \cap B(x,r/2)$ and $t \in (C_0^{-1}r,r/2)$ such that $\omega_p(y,t) \leq \varepsilon_0/2$.
    Assuming furthermore $h(r) \leq \min(\varepsilon_0/2, \varepsilonA)$, we can apply Proposition \ref{prop_omega} in $B(y,t)$ which yields
    \begin{equation*}
        \HH^{N-1}(K \cap B(y,t)) \geq \omega_{N-1}(1 - \varepsilon) r^{N-1},
    \end{equation*}
    and concludes the proof.
\end{proof}

We are going to deduce that the area is lower-semicontinuous along sequence of almost-minimizers.
\begin{corollary}[Lower semicontinuity of the area]\label{cor_semicontinuity_area}
    Let $(\Omega_i)_i$ and $\Omega$ be a sequence of open sets as in (\ref{eq_Omega}).
    Let $(u_i,K_i)_i$ be a sequence such that for all $i$, $(u_i,K_i)$ is an almost-minimizer with gauge $h_i$ in $\Omega_i$.
    If $(u_i,K_i)_i$ converges to a pair $(u,K)$ in $\Omega$ and
    \begin{equation*}
        \lim_{r \to 0} \limsup_{i \to +\infty} h_i(r) = 0,
    \end{equation*}
    then for all open set $V \subset \Omega$, we have
    \begin{equation*}
        \HH^{N-1}(K \cap V) \leq \liminf_{i \to +\infty} \HH^{N-1}(K_i \cap V).
    \end{equation*}
\end{corollary}
\begin{proof}
    We first deal with the case where $V \subset \subset \Omega$, in particular $V \subset \Omega_i$ for $i$ big enough.
    We show that the sequence $(K_i \cap V)_i$ is uniformly concentrated with respect to $K \cap V$ in the ambient space $V$, see Definition \ref{defi_uc}.
    Let $\varepsilon \in (0,1)$ and let $\varepsilon_0 > 0$ and $C_0 \geq 1$ be the associated constant given by Theorem \ref{prop_uc}.
    Let $x \in E \cap V$ and let us fix a radius $r(x)$ such that $B(x,2r(x)) \subset V$ and $\limsup_{i \to +\infty} h_i(r(x)) < \varepsilon_0$. In particular, for $i$ big enough $V \subset \Omega_i$ and $h_i(r(x)) \leq \varepsilon_0$.
    For $0 < r \leq r(x)$ and for $i$ big enough, there exists $x_i \in K_i$ such that $\abs{x_i - x} \leq r/2$ and thus $B(x_i,r/2) \subset B(x,r) \subset V \subset \Omega_i$.
    Since $h_i(r/2) \leq \varepsilon_0$, Theorem \ref{prop_uc} applied in $B(x_i,r/2) \subset \Omega_i$ shows that there exists $y_i \in K_i \cap B(x_i,r/4)$ and $t_i \in (C_0^{-1}r,r/4)$ such that
    \begin{equation*}
        \HH^{N-1}(K_i \cap B(y_i,t_i)) \geq \omega_{N-1} (1 - \varepsilon) t_i^{N-1}.
    \end{equation*}
    We also clearly have $B(y_i,t_i) \subset B(x_i,r/2) \subset B(x,r) \cap V$.
    Definition \ref{defi_uc} is thus satisfied and by (\ref{eq_uc}), we deduce that
    \begin{equation*}
        \HH^{N-1}(K \cap V) \leq \liminf_{i \to +\infty} \HH^{N-1}(K_i \cap V).
    \end{equation*}
    For a general open set $V \subset \Omega$, we consider an exhaustion of $V$ by open sets $(V^n)_n$ such that $V^n \subset \subset V$. Thus, for all $n$,
    \begin{equation*}
        \HH^{N-1}(K \cap V^n)   \leq \liminf_{i \to +\infty} \HH^{N-1}(K_i \cap V^n) \leq \liminf_{i \to +\infty} \HH^{N-1}(K_i \cap V)
    \end{equation*}
    whence by letting $n \to +\infty$,
    \begin{equation*}
        \HH^{N-1}(K \cap V) \leq \liminf_{i \to +\infty} \HH^{N-1}(K_i \cap V),
    \end{equation*}
    as desired.
\end{proof}

We finally prove Theorem \ref{thm_limit}.
\begin{proof}[Proof of Theorem \ref{thm_limit}]
    We let $(u_i,K_i)_i$ be a sequence such that for all $i$, $(u_i,K_i)$ is a topological almost-minimizer with gauge in $\Omega_i$. We assume that $(u_i,K_i)_i$ converges to a pair $(u,K)$ in $\Omega$ and that
    \begin{equation*}
        \lim_{t \to 0} \limsup_{i \to +\infty} h_i(t) = 0.
    \end{equation*}
    This assumption implies by Lemma \ref{lem_uniformAF} that $K$ is coral and it will also allow us to apply Corollary~\ref{cor_semicontinuity_area}.
    We know by Proposition \ref{prop_weak_limit} that for all open ball $B(x,r) \subset \subset \Omega$ such that $h^+(r) < \varepsilonA$ and for all topological competitor $(v,L)$ of $(u,K)$ in $B(x,r)$, we have
    \begin{multline}\label{eq_limit_comparison0}
        \limsup_{i \to + \infty} \left(\int_{B(x,r) \setminus K_i} \CC{e(u_i)} \dd{x} + \HH^{N-1}(K_i \cap \overline{B}(x,r))\right) \\\leq \int_{B(x,r) \setminus L} \CC{e(v)} \dd{x} + \HH^{N-1}(L \cap \overline{B}(x,r)) + h^+(r) r^{N-1},
    \end{multline}
    where $h^+(r) := \lim_{t \to r^+} \limsup_{i \to +\infty} h_i(t)$.
    As the elastic energy and the area are lower semicontinuous by Lemma \ref{lem_semicontinuity_energy} and Corollary \ref{cor_semicontinuity_area}, we have in particular
    \begin{multline}\label{eq_limit_comparison}
        \int_{B(x,r) \setminus K} \CC{e(u)} \dd{x} + \HH^{N-1}(K \cap B(x,r)) \\\leq \int_{B(x,r) \setminus L} \CC{e(v)} \dd{x} + \HH^{N-1}(L \cap \overline{B}(x,r)) + h^+(r) r^{N-1}.
    \end{multline}
    Here, it is in fact possible to remove the closure $\overline{B}(x,r)$ at the right-hand side. Indeed, for all small $t > r$ such that $h^+(t) < \varepsilonA$, the pair $(v,L)$ is still a topological competitor of $(u,K)$ in $B(x,t)$ so the energy comparison (\ref{eq_limit_comparison}) holds when one replaces $B(x,r)$ by $B(x,t)$. Then one can let $t \to r^+$ and use the fact that $K \cap \partial B(x,r) = L \cap \partial B(x,r)$ to deduce the inequality with $B(x,r)$ on both sides (without closure).
    Here we have also used the fact that $h^+$ is right-continuous.
    It is clear that if we set $h(r) = h^+(r)$ when $h^+(r) < \varepsilon$ and $h(r) = +\infty$ otherwise, then $(u,K)$ is an almost-minimizer in $\Omega$ with gauge $h$.

    It is left to check that for all open ball $B(x,r) \subset \subset \Omega$, we have
    \begin{equation*}
        \limsup_{i \to +\infty} \int_{B(x,r) \setminus K_i} \CC{e(u_i)} \dd{x} \leq \int_{B(x,r) \setminus K} \CC{e(u)} \dd{x} + h(r) r^{N-1}
    \end{equation*}
    and
    \begin{equation*}
        \limsup_{i \to +\infty} \HH^{N-1}(K_i \cap \overline{B}(x,r)) \leq \HH^{N-1}(K \cap \overline{B}(x,r)) + h(r) r^{N-1}.
    \end{equation*}
    We can directly assume that $h^+(r) < \varepsilonA$ so that $h(r) = h^+(r)$.
    Observe that for $t > r$ slightly bigger than $r$ such that $\overline{B}(x,t) \subset \Omega$ and $h^+(t) < \varepsilon$, we have $B(x,t) \subset \Omega_i$ and $h_i(t) \leq \varepsilonA$ for $i$ big enough and this implies a uniform bound by (\ref{eq_AF2}),
    \begin{equation*}
        \int_{B(x,t) \setminus K_i} \abs{e(u_i)}^2 \dd{x} + \HH^{N-1}(K_i \cap B(x,t)) \leq C t^{N-1}.
    \end{equation*}
    This makes sure that we will always deal with finite quantities in the argument below.
    We apply (\ref{eq_limit_comparison0}) with $(u,K)$ being a competitor of itself in $B(x,r)$, and we obtain,
    \begin{multline}\label{eq_limsupinf}
        \limsup_{i \to +\infty} \left(\int_{B(x,r) \setminus K_i} \CC{e(u_i)} \dd{x} + \HH^{N-1}(K_i \cap \overline{B}(x,r))\right) \\\leq \int_{B(x,r) \setminus K} \CC{e(u)} \dd{x} + \HH^{N-1}(K \cap \overline{B}(x,r)) + h^+(r) r^{N-1}.
    \end{multline}
    We first deal with the limit superior of $\HH^{N-1}(K_i \cap \overline{B}(x,r))$. It follows from (\ref{eq_limsupinf}) that
    \begin{multline*}
        \left(\liminf_{i \to +\infty} \int_{B(x,r) \setminus K_i} \CC{e(u_i)} \dd{x}\right) + \left(\limsup_{i \to +\infty} \HH^{N-1}(K_i \cap \overline{B}(x,r))\right) \\\leq \int_{B(x,r) \setminus K} \CC{e(u)} \dd{x} + \HH^{N-1}(K \cap \overline{B}(x,r)) + h^+(r) r^{N-1},
    \end{multline*}
    and we know by lower semicontinuity of the energy that
    \begin{equation*}
        \liminf_{i \to +\infty} \int_{B(x,r) \setminus K_i} \CC{e(u_i)} \dd{x} \geq \int_{B(x,r) \setminus K} \CC{e(u)} \dd{x}
    \end{equation*}
    so we deduce that
    \begin{equation*}
        \limsup_{i \to +\infty} \HH^{N-1}(K_i \cap \overline{B}(x,r)) \leq \HH^{N-1}(K \cap \overline{B}(x,r)) + h^+(r) r^{N-1}.
    \end{equation*}
    We pass to the limit superior of $\int_{B(x,r) \setminus K_i} \CC{e(u_i)} \dd{x}$.
    It follows from (\ref{eq_limsupinf}) that
    \begin{multline*}
        \left(\limsup_{i \to +\infty} \int_{B(x,r) \setminus K_i} \CC{e(u_i)} \dd{x}\right) + \left(\liminf_{i \to +\infty} \HH^{N-1}(K_i \cap \overline{B}(x,r))\right) \\\leq \int_{B(x,r) \setminus K} \CC{e(u)} \dd{x} + \HH^{N-1}(K \cap \overline{B}(x,r)) + h^+(r) r^{N-1}.
    \end{multline*} 
    By an argument which is now usual, this also holds when one replaces $B(x,r)$ by balls $B(x,t)$ where $t$ is a radius slightly bigger than $r$ such that $\overline{B}(x,t) \subset \Omega$ and $h^+(t) < \varepsilonA$. We know by lower semicontinuity of the area that
    \begin{equation*}
        \liminf_{i \to + \infty} \HH^{N-1}(K_i \cap \overline{B}(x,t)) \geq \HH^{N-1}(K \cap B(x,t))
    \end{equation*}
    so we deduce that
    \begin{multline*}
        \left(\limsup_{i \to +\infty} \int_{B(x,t) \setminus K_i} \CC{e(u_i)} \dd{x}\right) + \HH^{N-1}(K \cap B(x,t)) \\\leq \int_{B(x,t) \setminus K} \CC{e(u)} \dd{x} + \HH^{N-1}(K \cap \overline{B}(x,t)) + h^+(t) t^{N-1}
    \end{multline*}
    and in particular
    \begin{multline*}
        \left(\limsup_{i \to +\infty} \int_{B(x,r) \setminus K_i} \CC{e(u_i)} \dd{x}\right) + \HH^{N-1}(K \cap \overline{B}(x,r)) \\\leq \int_{B(x,t) \setminus K} \CC{e(u)} \dd{x} + \HH^{N-1}(K \cap \overline{B}(x,t)) + h^+(t) t^{N-1}.
    \end{multline*}
    Then one can let $t \to r^+$ and use $K \cap \partial B(x,r) = L \cap \partial B(x,r)$ to deduce
    \begin{equation*}
        \limsup_{i \to + \infty} \int_{B(x,r) \setminus K_i} \CC{e(u_i)} \dd{x} \leq \int_{B(x,r) \setminus K} \CC{e(u)} \dd{x} + h^+(r) r^{N-1}. 
    \end{equation*}
\end{proof}

\section{Applications}

\subsection{Existence of blow-up limits}\label{section_blowup}

We adapt the notion of global minimizers introduced by {\sc Bonnet} \cite{Bonnet} from the Mumford-Shah to the Griffith setting.
\begin{definition}[Global minimizer]
    A Griffith global minimizer in $\R^N$ is a coral pair $(u,K)$ such that for all $x \in K$, for all $r > 0$ and for all topological competitor $(v,L)$ of $(u,K)$ in $B = B(x,r)$, we have
    \begin{equation*}
        \int_B \CC{e(u)} \dd{x} + \HH^{N-1}(K \cap B) \leq \int_B \CC{e(v)} \dd{x} + \HH^{N-1}(L \cap B).
    \end{equation*}
\end{definition}
This notion draws its importance from the fact that blow-up limits of Griffith minimizer are global minimizers.
We will justify this soon but let us first describe the known (or expected) global minimizers.
The first example of global minimizers are those for which $u$ is piecewise rigid. In this case, $K$ is a minimal set of codimension 1 in $\R^N$ and a partial classification is known.
In dimension $N = 2$, there are exactly two possibilities: a line or a triple junction (three half lines meeting with an angle $2\pi/3$). 
In dimension $N = 3$, there are exactly three possibilities: an hyperplane, a $\mathbb{Y}$ cone (three half-planes meeting with an angle $2\pi/3$) or a $\mathbb{T}$ cone (the cone over the edges of a regular tetrahedron). 
We refer to {\sc Taylor} \cite{ta} or {\sc David} \cite[Theorem 1.9]{David2009} for a proof.
As soon as $N \geq 4$, a few examples are known but not the full classification. There is for example the cone over the $(N-2)$-skeleton of a cube \cite{Brakke} and the cone over the $(N-2)$-skeleton of a regular simplex \cite{LawMor}.
What about the global minimizers for which $u$ is not piecewise rigid, we expect crack-tips in the plane ($K$ is a half-line) and crack-fronts in higher dimensions ($K$ is a half-hyperplane).
This was proved by {\sc David} and {\sc Bonnet} \cite{BonnetDavid} in the scalar case.
It is not known if there could be other kind of global minimizers.

Let us describe now the blow-up limit procedure.
Let $(u,K)$ be a topological almost-minimizer in $\Omega$ with a gauge $h$. We recall in this case that the gauge satisfies $\lim_{r \to 0} h(r) = 0$, see Definition \ref{defi_quasi}.
We fix $x_0 \in K$.
We consider a sequence of radii $(r_i)_i$ such that $r_i \to 0$ and for each $i$, we consider the pair $(u_i,K_i)_i$ defined by
\begin{equation*}
    u_i(x) = r_i^{-1/2} u(x_0 + r_i x) \quad \text{and} \quad K_i := r_i^{-1} (K - x_0)
\end{equation*}
in the domain $\Omega_i = r_i^{-1}(\Omega - x_0)$.
This is a topological almost-minimizer in $\Omega_i$ with gauge $h_i(t) = h(r_i t)$.
Since $\Omega$ is an open set, one can see that
\begin{equation*}
    \text{for all compact set $H \subset \R^N$, we have $H \subset \Omega_i$ for $i$ large enough.}
\end{equation*}
If the sequence $(u_i,K_i)_i$ converges to a pair $(u_\infty,K_\infty)$ in $\R^N$, then we call $(u_\infty,K_{\infty})$ a \emph{blow-up limit} of $(u,K)$ at $x_0$. 
One can see that the limit gauge is identically zero because $h_i(t) \to 0$ for all $t \geq 0$.
Therefore, a blow-up limit is a global minimizer in $\R^N$ by Theorem \ref{thm_limit}.

Note that we can always extract a subsequence such that the pairs $(u_i,K_i)_i$ above converge.
Indeed, for all $R > 0$ and for $i$ big enough such that $B(0,R) \subset \Omega_i$ and $h(r_i R) \leq \varepsilonA$, we have by Ahlfors-regularity (\ref{eq_AF2})
\begin{equation*}
    \int_{B(0,R)} \abs{e(u_i)}^2 \dd{x} = r_i^{1-N} \int_{B(x_0, r_i R)} \abs{e(u)}^2 \dd{x} \leq C R^{N-1}.
\end{equation*}
Therefore, we can apply Lemma \ref{lem_compactness} to extract a convergent subsequence.
This shows that every point has blow-up limits but it is not known whether there is uniqueness.

We observe that if $(u_{\infty},K_\infty)$ is a blow-up limit of $(u,K)$ at $x_0$, then for all $r > 0$, the rescaled pair $(v_{\infty},L_{\infty})$ defined by
\begin{equation}\label{eq_rescaled_blowup}
    v_{\infty}(x) = r^{-1/2} u_{\infty}(r x) \quad \text{and} \quad L_{\infty} = r^{-1} K_\infty
\end{equation}
is also a blow-up limit of $(u,K)$ at $x_0$. This is a direct application of Remark \ref{rmk_limit_scaling}.
If $K$ has a unique blow-up limit at $x_0$, it must therefore be a cone centered at $0$.
In Proposition~\ref{prop_classification} we will classify the possible global minimizers whose crack set is a cone when $N = 2$.

\begin{proposition}\label{prop_blowup}
    Let $(u,K)$ be a topological almost-minimizer in $\Omega$ with gauge $h$. Let $x_0 \in K$ and let a sequence $(r_i)_i$ such that $r_i \to 0$ and such that the pairs
    \begin{equation*}
        u_i(x) = r_i^{-1/2} u(x_0 + r_i x) \quad \text{and} \quad K_i := r_i^{-1} (K - x_0)
    \end{equation*}
    converge to a pair $(u_{\infty},K_{\infty})$ in $\R^N$. Then $(u_{\infty},K_{\infty})$ is a global minimizer in $\R^N$.
    Moreover, for all open ball $B = B(y,t) \subset \R^N$, we have
    \begin{equation}\label{eq_blowup_energy}
        \int_{B(y,t)} \CC{e(u_{\infty})} \dd{x} = \lim_{i \to +\infty} r_i^{1-N} \int_{B(x_0 + r_i y, r_i t)} \CC{e(u)} \dd{x}
    \end{equation}
    and
    \begin{align*}
        \HH^{N-1}(K_{\infty} \cap B(y,t)) &\leq \liminf_{i \to +\infty} r_i^{1-N} \HH^{N-1}(K_i \cap B(x_0 + r_i y, r_i t))\\
        \HH^{N-1}(K_{\infty} \cap \overline{B}(y,t)) &\geq \limsup_{i \to +\infty} r_i^{1-N} \HH^{N-1}(K_i \cap B(x_0 + r_i y, r_i t)).
    \end{align*}
    If in addition $\lim_{r \to 0} \omega_2(x_0,r) = 0$, then $u$ is piecewise rigid and $K$ is a minimal set in $\R^N$.
\end{proposition}
\begin{proof}
    This is a direct application of Theorem \ref{thm_limit}.
\end{proof}

We finally investigate the possible global minimizers $(u,K)$ in the plane when $K$ is a cone.
In the setting of Mumford-Shah global minimizers, a similar classification is due to Bonnet \cite{Bonnet} under the more general assumption that $K$ is connected. Such a result is not yet available for Griffith due to the lack of analogue of the Bonnet monotonicity formula.

\begin{proposition}\label{prop_classification}
    Assume that $\C e =\lambda tr(e) Id + 2\mu e$ where $\lambda$ and $\mu$ are the Lam\'e coefficients satisfying $\mu > 0$ and $\mu + \lambda > 0$. Let $(u,K)$ be a global minimizer in $\R^2$ and assume that $K$ is a cone centered centered at $0$. Then, either $K$ is empty, a line, a triple junction or a half-line (crack-tip).
\end{proposition}
\begin{proof}

    Let $(u,K)$ be a global minimizer in $\R^2$ and assume that $K$ is a cone centered at $0$.
    Since the density of $K$ at $0$ is bounded by Ahlfors-regularity, this cone can only be composed of a finite number of half-lines. We can directly assume that $K$ is composed of at least two half lines, as the other cases are already described in the conclusion of the proposition.
    Thus, $\R^2 \setminus K$ is composed of infinite angular sectors with aperture in $(0,2\pi)$.
    It is standard that by taking outer variations, one obtains that $u$ is a weak solution of the Lam\'e system: denoting the strain by  $\sigma :=  \C e(u) =\lambda {\rm tr}(e(u)) Id + 2\mu e(u)$, then in each connected component $\Omega$ of $\R^2\setminus K$, we have in a weak sense
    $$
    \left\{
        \begin{array}{cc}
            {\rm div}(\sigma) =0 &\text{ in }\Omega \\ 
            \sigma\cdot \nu = 0 & \text{ on } \partial \Omega.
        \end{array}
    \right.
    $$
    In each angular domain of $\R^2 \setminus K$, we invoke the regularity theory for the Lamé system in polygonal domains.
    More precisely we shall use  \cite[Theorem 3.11 (Decay Estimate I)]{epitaxially})  and deduce that there exists $C_0,\alpha>0$ for which the following decay property holds:
    \begin{equation}\label{eq_decay0}
        \int_{B(0,r) \setminus K} \abs{\nabla u}^2 \dd{x} \leq C_0 r^{1+\alpha} \int_{B(0,1) \setminus K}\left( |u|^2 + \abs{\nabla u}^2 \right) \dd{x}, \quad \forall 0 < r < 1.
    \end{equation}
    Note that from the definition of pairs, we just have $u \in W^{1,2}_{\rm loc}(\R^2 \setminus K;\R^2)$ but since each connected component of $B(0,1) \setminus K$ is a Lipschitz domain, the Korn-Poincaré inequality shows that we actually have $u \in W^{1,2}(B(0,1) \setminus K;\R^2)$. Therefore the constant
    \begin{equation*}
        C_1 := C_0 \int_{B(0,1)}\left( |u|^2 + \abs{\nabla u}^2 \right) \dd{x}
    \end{equation*}
    is finite and we can reformulate (\ref{eq_decay0}) as
    \begin{equation}\label{eq_decay}
        \int_{B(0,r) \setminus K} \abs{\nabla u}^2 \dd{x} \leq C_1 r^{1+\alpha}, \quad \forall \ 0 < r < 1.
    \end{equation}
    Now we proceed to a blow-up procedure: from the pair $(u,K)$ we define $(u_n,K_n)_n$ as being the blow-up sequence 
    $$K_n:=\frac{1}{r_n}K \quad \quad \text{and} \quad u_n(x)=r_n^{-1/2}u(r_n x),$$
    with $r_n=1/n \to 0$. We can extract a subsequence which converges to a pair $(u_{\infty},K_{\infty})$ and from Proposition \ref{prop_blowup}, we know that $(u_\infty,K_\infty)$ is still a global minimizer in the plane. Of course since $K$ is assumed to be a cone, it holds $K_\infty=K$. Now we want to prove that $e(u_\infty)=0$. For that purpose, we apply (\ref{eq_decay}) and \eqref{eq_blowup_energy} from Proposition~\ref{prop_blowup} to deduce that for any given $a>0$,
    $$\int_{B(0,a) \setminus K} |e(u_\infty)|^2 \;dx =\lim_{n \to +\infty} \frac{1}{r_n}\int_{B(0,r_n a) \setminus K} |e(u)|^2 \dd x \leq C_1 a^{1+\alpha}  \lim_{n\to +\infty}r_n^{\alpha} =0,$$
    thus $e(u_\infty)=0$ in $B(0,a)$.
    Since $a>0$ is arbitrary, this shows that $e(u_\infty)=0$ everywhere in $\R^2 \setminus K$. But then $(u_\infty, K_\infty)$ is a global minimizer with $e(u_{\infty})=0$, so $K_\infty$ is a minimal set in $\R^2$. In virtue of \cite[Theorem 10.1]{David2009}, we conclude that $K$ must be a line or a triple junction.
\end{proof}

\subsection{Equivalent definitions of the singular part}

Let $(u,K)$ be a topological almost-minimizer with gauge $h$ in $\Omega$.
We define the \emph{regular part} of $K$ as the set of points $x \in K$ for which there exists a sequence $(r_i)_i$ going to $0$ and an hyperplane $P$ passing through $x$ such that
\begin{equation*}
    \lim_{i \to +\infty} r_i^{-1} \left(\sup_{y \in P \cap B(x,r_i)} \mathrm{dist}(y,K) + \sup_{y \in K \cap B(x,r_i)} \mathrm{dist}(y,P)\right) = 0.
\end{equation*}
We define the \emph{singular part} of as the set of non-regular points of $K$, denoted by the symbol $\Sigma(K)$.

If the gauge $h$ is decaying as power, we expect that regular points are equivalently characterized as points $x \in K$ in the neighborhood of which $K$ is a smooth hypersurface. Our definition of ``regular points" is therefore quite weak but it is motivated by its application in Proposition \ref{prop_dimension}, where the dimension of $\Sigma(K)$ will be controlled by the integrability exponent of $e(u)$. This control was first observed for Mumford-Shah minimizers by {\sc Ambrosio}, {\sc Fusco} and {\sc Hutchinson} \cite{AFH}. As we lack $\varepsilon$-regularity theorems providing an equivalence between all reasonable definitions of the regular part, we need to choose a large definition of ``regular points" in order to adapt their result.


We now investigate different equivalent characterization of regular points.
One can already see that regular points are characterized by the condition
\begin{equation*}
    \liminf_{r \to 0} \beta^{\rm bil}(x,r) = 0,
\end{equation*}
where the bilateral flatness $\beta^{\rm bil}$ is defined in (\ref{eq_bilateral_flatness}).
The goal of the rest of this section is to justify that they are also characterized by the condition
\begin{equation*}
    \liminf_{r \to 0} \beta(x,r) + \omega(x,r) = 0.
\end{equation*}

We show first that the bilateral flatness is controlled by the flatness and the normalized elastic energy.
\begin{proposition}\label{prop_regular0}
    There exists a constant $\varepsilon_0 > 0$ (which depends only on $N$ and $\C$) such that the following holds.
    Let $(u,K)$ be a topological almost-minimizer with gauge $h$ in an open set $\Omega$. Then for all $x_0 \in K$ and $r_0 > 0$ such that $B(x_0,r_0) \subset \Omega$ and $h(r_0) \leq \varepsilon_0$, we have
    \begin{equation*}
        \beta^{\rm bil}(x_0,r_0/2) \leq C \left(\beta(x_0,r_0) + \omega_2(x_0,r_0)^{1/(6m)}\right),
    \end{equation*}
    where $m = N-1$ and $C \geq 1$ is some constant which depends only on $N$ and $\mathbb{C}$.
\end{proposition}
\begin{proof}
    The letter $C\geq 1$ denotes a generic constant   which depends only on $N$ and $\C$.
    As usual, we assume that $B(x_0,r_0) = B(0,1)$.
    We let $\tau_0 \in (0,1/8)$ denote the constant of Lemma \ref{lem_Jinit} for $p = 2$ and $M = 1$ (it depends only on $N$ and $\C$).
    If $\beta(0,1) \geq 1/8$, the inequality holds trivially because we always have $\beta^{\rm bil}(0,1/2) \leq 1$.
    Otherwise, we let $P_0$ denote an hyperplane which achieves the infimum in the definition of $\beta(0,1)$ and we choose a unit normal $\nu_0$ to $P_0$. We let $a_1$, $a_2$ be defined as in the beginning of Section \ref{section_jump}
    Then we apply Lemma \ref{lem_slicing} with $\varepsilon = 1/8$ and this shows that for all unit vector $\nu \in \mathbf{S}^{N-1}$ such that $\abs{\nu - \nu_0} \leq 1/8$, we have
    \begin{equation*}
        J(\nu) \HH^{N-1}\left(S(\nu)\right)^2 \leq C \omega_1(0,1)^{1/2},
    \end{equation*}
    where $S(\nu)$ is the set of holes through slicing $$S(\nu) := P \cap B(0,1/2) \setminus \pi_P(K \cap B(0,1)),$$ $P$ is the hyperplane $x_0 + \nu^\perp$, $\pi_P$ is the orthogonal projection onto $P$, and
    $$J(\nu) := \abs{b \cdot \nu} + \abs{A \nu},$$
    where $b := b_1 - b_2$, $A := A_1 - A_2$. For any parameter $\varepsilon \in (\beta(0,1),1/8)$, Lemma \ref{lem_int_rigid2} shows that
    \begin{equation*}
        \int_{\nu \in \mathbf{S}^{N-1} \cap B(\nu_0,\varepsilon)} \abs{b \cdot x} + \abs{A x} \dd{\HH^{N-1}(x)} \geq C^{-1} \varepsilon^{N-1} \left(\abs{b} + |A|\right).
    \end{equation*}
    This allows to find a vector $\nu \in \mathbf{S}^{N-1}$ such that $\abs{\nu - \nu_0} \leq \varepsilon$ and $J(\nu) \geq C^{-1} \varepsilon^{N-1} J(0,1)$, and thus
    \begin{equation}\label{eq_Pcontrol}
        \HH^{N-1}\left(S(\nu)\right)^2 \leq C \varepsilon^{1-N} \omega_2(0,1)^{1/2}.
    \end{equation}
    Note that since
    \begin{equation*}
        K \cap B(0,1) \subset \set{\abs{x \cdot \nu_0} \leq \beta(0,1)}
    \end{equation*}
    and $\beta(0,1) \leq \varepsilon$ as well as $\abs{\nu - \nu_0} \leq \varepsilon$, we have
    \begin{equation}\label{eq_bil1}
        K \cap B(0,1) \subset \set{\abs{x \cdot \nu} \leq 2 \varepsilon}
    \end{equation}
    so the points of $K \cap B(0,1)$ are at distance $\leq 2 \varepsilon$ from $P$.
    Then we use (\ref{eq_Pcontrol}) to evaluate how far are the points of $P \cap B(0,1/2)$ from $K$.
    For $x \in P \cap B(0,1/2)$, we are going to prove that
    \begin{equation}\label{eq_bil2}
        \mathrm{dist}(x,K) \leq \max \left(4\varepsilon, C\varepsilon^{-1/2} \omega_2(0,1)^{1/(4m)}\right),
    \end{equation}
    where $m = N-1$.
    For this we consider a radius $t > 0$ such that $B(x,t) \cap K \ne \emptyset$. We want to bound $t$ from above by the right-hand side of (\ref{eq_bil2}) and for this we can directly consider the case where $t > 4\varepsilon$.
    We see that $K \cap B(0,1) \subset \set{\abs{x \cdot \nu} \leq t/2}$ so it is not possible for $P \cap B(x,t/2)$ to contain a point of $\pi_P(K \cap B(0,1))$ and therefore by (\ref{eq_Pcontrol}),
    \begin{equation*}
        \HH^{N-1}\left(P \cap B(0,1/2) \cap B(x,t/2)\right)^2 \leq C \varepsilon^{1-N} \omega_2(0,1)^{1/2}.
    \end{equation*}
    On the other hand, since $0 \in K$ and $B(x,t) \cap K = \emptyset$, we have at most $t \leq 1/2$ so we can bound from below
    \begin{equation*}
        \HH^{N-1}\left(P \cap B(0,1/2) \cap B(x,t/2)\right) \geq C^{-1} t^{N-1}.
    \end{equation*}
    This proves our claim. In view of (\ref{eq_bil1}) and (\ref{eq_bil2}), we conclude that for all $\varepsilon \in (\beta(0,1),1/8)$, we have
    \begin{equation*}
        \beta^{\rm bil}(0,1/2) \leq C \max\left(\varepsilon, \varepsilon^{-1/2} \omega_2(0,1)^{1/(4m)}\right).
    \end{equation*}
    If $\beta(0,1) \leq \omega^{1/(6m)}$, we take $\varepsilon = \omega^{1/(6m)}$ and otherwise we take $\varepsilon = \beta(0,1)$.
    In both case, this shows that
    \begin{equation*}
        \beta^{\rm bil}(0,1/2) \leq C \left(\beta(0,1) + \omega_2(0,1)^{1/(6m)}\right),
    \end{equation*}
    as desired.\end{proof}

Reciprocally, a blow-up type argument shows that the bilateral flatness controls the normalized elastic energy.
\begin{proposition}\label{prop_regular1}
    For all $\varepsilon > 0$, there exists $\varepsilon_0 > 0$ and $\gamma \in (0,1)$ (depending on $N$, $\C$ and $\varepsilon$) such that the following holds.
    Let $(u,K)$ be a topological almost-minimizer with gauge $h$ in an open set $\Omega$. If $x_0 \in K$ and $r_0 > 0$ are such that $B(x_0,r_0) \subset \Omega$ and
    \begin{equation*}
        \beta^{\rm bil}(x_0,r_0) + h(r_0) \leq \varepsilon_0,
    \end{equation*}
    then
    \begin{equation*}
        \beta(x_0, \gamma r_0) + \omega(x_0, \gamma r_0) \leq \varepsilon.
    \end{equation*}
\end{proposition}
\begin{proof}
    We let $C\geq 1$ denote a generic constant  which depends only on $N$ and $\C$.
    As usual, we assume that $B(x_0,r_0) = B(0,1)$.
    We consider a fixed $\varepsilon > 0$ and we assume that for all choice of constant $\varepsilon_0 > 0$, the statement does not hold.
    Therefore, we can find a sequence $(r_i)_i \in (0,1)$ going to $0$ and sequence of topological almost-minimizers $(u_i,K_i)_i$ with gauge $h_i$ in $B(0,1)$ such that for all $i$,
    \begin{equation*}
        \lim_{i \to +\infty} \beta^{\rm bil}_{K_i}(0,1) + h_i(1) \leq r_i^2
    \end{equation*}
    but
    \begin{equation*}
        \beta(0,r_i) + r_i^{1-N} \int_{B(0,r_i)} \abs{e(u_i)}^2 \dd{x} \geq \varepsilon. 
    \end{equation*}
    Note that by the scaling property of the flatness and the fact that $\beta_{K_i}(0,1) \leq \beta^{\rm bil}_{K_i}(0,1) \leq r_i^2$, we have
    \begin{equation*}
        \beta_{K_i}(0,r_i) \leq r_i^{-1} \beta_{K_i}(0,1) \leq r_i.
    \end{equation*}
    We extract a subsequence such that for all $i$, $\beta(0,r_i) \leq \varepsilon/2$ and thus we have to contradict the fact that for all $i$, $\omega(0,r_i) \geq \varepsilon/2$.
    We also extract a subsequence such that for all $i$, $h_i(1) \leq \varepsilonA$ and thus by (\ref{eq_AF2}), for all ball $B(x,r) \subset B(0,1)$,
    \begin{equation}\label{eq_limit_AF}
        \sup_i \left(\int_{B(x,r)} \abs{e(u_i)}^2 \dd{x} + \HH^{N-1}(K_i \cap B(x,r))\right) \leq C r^{N-1}.
    \end{equation}
    We define a pair $(v_i,L_i)$ in $B(0,r_i^{-1})$ by
    \begin{equation*}
        v_i(x) := r_i^{-1/2} u_i(r_i x) \quad \text{and} \quad L_i := r_i^{-1} K_i.
    \end{equation*}
    We observe that $(v_i,L_i)$ is a topological almost-minimizer in $B(0,r_i^{-1})$ with gauge $\tilde{h}_i(t) = h_i(r_i t)$.
    We also observe that for all $R > 0$ and for $i$ big enough such that $B(0,R) \subset B(0,r_i^{-1})$, we have by (\ref{eq_limit_AF}),
    \begin{equation*}
        \int_{B(0,R)} \abs{e(v_i)}^2 \dd{x} = r_i^{1-N} \int_{B(0,r_i R)} \abs{e(u_i)}^2 \dd{x} \leq C R^{N-1}
    \end{equation*}
    which is bounded. It follows that we can extract a subsequence of $(v_i,L_i)_i$ which converges to a pair $(v_{\infty},L_{\infty})$ in $\R^N$. As $h_i(1) \to 0$, we have $\tilde{h}_i(t) \to 0$ for all $t \geq 0$ and thus the limit gauge $h$ is identically zero.
    By application of Theorem \ref{thm_limit}, the pair $(v_{\infty},L_{\infty})$ is a global minimizer in $\R^N$ and we have for all $R > 0$,
    \begin{equation}\label{eq_limit_energy}
        \int_{B(0,R)} \CC{e(v_{\infty})} \dd{x} = \lim_{i \to +\infty} r_i^{1-N} \int_{B(0,r_i R)} \CC{e(u_i)} \dd{x}.
    \end{equation}
    We also observe that for all $R > 0$ and for all $i$ big enough such that $B(0,R) \subset B(0,r_i^{-1})$,
    \begin{equation*}
        \beta^{\rm bil}_{L_i}(0,R) = \beta^{\rm bil}_{K_i}(0,r_i R) \leq (r_i R)^{-1} \beta^{\rm bil}_{K_i}(0,1) \leq r_i R^{-1}
    \end{equation*}
    whence $\beta^{\rm bil}_{L_\infty}(0,R) = 0$. This means that $L_{\infty}$ coincides with an hyperplane in $B(0,R)$ and as $R$ is arbitrarily large, we deduce that $L_{\infty}$ coincides with an hyperplane in $\R^N$.
    By testing the minimality condition of $(v_{\infty},L_{\infty})$ with outer variations of the form $(v_{\infty} + \varepsilon \varphi,L_{\infty})$, where $\varphi \in C^1_c(\R^N;\R^N)$, we see that $v_{\infty}$ solves in a weak sense the elliptic PDE $\mathrm{div}(\C e(v_{\infty})) = 0$ in the complement of $L_{\infty}$ with a Neumann boundary condition $\C e(v_{\infty}) \cdot e_N = 0$ on each side of $L_{\infty}$. By elliptic regularity, it follows that there exists a constant $C \geq 1$ such that for all $R > 1$,
    \begin{equation*}
        \int_{B(0,1)} \abs{e(v_{\infty})}^2 \dd{x} \leq \frac{C}{R^N} \int_{B(0,R)} \abs{e(v_{\infty})}^2 \dd{x}.
    \end{equation*}
    But by (\ref{eq_limit_energy}), we have
    \begin{align*}
        \int_{B(0,R)} \CC{e(v_{\infty})} \dd{x} = \lim_{i \to +\infty} r_i^{1-N} \int_{B(0,r_i R)} \CC{e(u_i)} \dd{x} \leq C R^{N-1}
    \end{align*}
    so
    \begin{equation*}
        \int_{B(0,1)} \abs{e(v_{\infty})}^2 \dd{x} \leq C R^{-1}
    \end{equation*}
    and since $R > 0$ is arbitrarily large, we arrive at $\int_{B(0,1)} \abs{e(v_{\infty})}^2 \dd{x} = 0$.
    Using (\ref{eq_limit_energy}) again, this gives
    \begin{align*}
        \lim_{i \to +\infty} r_i^{1-N} \int_{B(0,r_i)} \CC{e(u_i)} \dd{x} = 0
    \end{align*}
    and contradicts the assumption.
\end{proof}

\begin{corollary}
    Let $(u,K)$ be an almost-minimizer in $\Omega$ with gauge $h$.
    For all $x \in K$, $x$ is a regular point of $K$ if and only if
    \begin{equation}\label{eq_def_regular2}
        \liminf_{r \to 0} \beta(x,r) + \omega(x,r) = 0.
    \end{equation}
\end{corollary}
\begin{proof}
    Let $x \in K$ be a regular point, i.e., there exists a sequence $(r_i)_i$ going to $0$ such that $\lim_{r \to 0} \beta^{\rm bil}(x,r_i) = 0$. For all $k \geq 0$, Proposition (\ref{prop_regular1}) shows that there exist $\varepsilon_k > 0$ and $c_k \in (0,1)$ such that for all $r > 0$ with $B(x,r) \subset \Omega$, if $\beta^{\rm bil}(x,r) + h(r) \leq \varepsilon_k$, then $\beta(x,c_k r) + \omega(x,c_k r) \leq 2^{-k}$.
    But for all $k \geq 0$, we always have $\beta^{\rm bil}(x,r_i) + h(r_i) \leq \varepsilon_k$ for $i$ big enough and thus $\beta(x, c_k r_i) + \omega(c_k r_i) \leq 2^{-k}$ for $i$ big enough. We deduce that
    \begin{equation*}
        \liminf_{r \to 0} \beta(x,r) + \omega(x,r) \leq 2^{-k},
    \end{equation*}
    but since $k$ is arbitrarily large, $\liminf_{r \to 0} \beta(x,r) + \omega(x,r) = 0$.
    Reciprocally, it directly follows from Proposition \ref{prop_regular0} that the condition (\ref{eq_def_regular2}) implies $\liminf_{r \to 0} \beta^{\rm bil}(x,r) = 0$.
\end{proof}

\subsection{Dimension of the singular part}

In the scalar case, {\sc Ambrosio}, {\sc Fusco} and {\sc Hutchinson} \cite{AFH} established that if $\nabla u$ is integrable with an exponent $p > 2$, then $K$ is smooth out of a subset of dimension less or equal to $\max(N-2,N-p/2)$. {\sc De Lellis} and {\sc Focardi} \cite[Proposition 5]{DLF1} furthermore proved that a sharp $L^p$ estimate $\nabla u \in L^{4,\infty}$ was equivalent to a variant of the Mumford-Shah conjecture.
We are going to use Theorem \ref{thm_limit} to adapt \cite[Corollary 5.7]{AFH} to Griffith almost-minimizers but our proof is only a minor variation. The existence of an integrability exponent $p > 2$ has been established in \cite[Theorem 2.4]{FLK} for Griffith minimizers in the plane, following the method of {\sc De Philippis} and {\sc Figalli} \cite{DPF}. 

\begin{proposition}\label{prop_dimension}
    Let $(u,K)$ be an almost-minimizer in $\Omega$ with gauge $h$. If there exists $p > 2$ such that $e(u) \in L^p_{\rm loc}(\Omega;\R^{N \times N})$, then
    \begin{equation*}
        \mathrm{dim}_{\mathcal{H}}(\Sigma(K)) \leq \max(N-2,N-p/2).
    \end{equation*}
\end{proposition}
\begin{proof}
    \emph{Step 1. We show that $$\text{the set $\set{x \in K | \limsup_{r \to 0} \omega(x,r) > 0}$ has a dimension $\leq N - p/2$},$$ where in the case $N - p/2 < 0$, this means that the set is empty.}
    We start with a general fact about locally integrable function which is that if $v \in L^p_{\rm loc}(\Omega)$ for some $p \geq 1$, then for all real number $s < N$, we have for $\HH^{N - p(N-s)}$-a.e. $x \in \Omega$,
    \begin{equation*}
        \lim_{r \to 0} r^{-s} \int_{B(x,r)} v \dd{y} = 0,
    \end{equation*}
    where in the case $N - p(N-s) < 0$, this means that the limit holds everywhere.
    Applying this in particular to $v = \abs{e(u)}^2 \in L^{p/2}_{\rm loc}(\Omega)$ and $s = N-1$, we see that for $\HH^{N-p/2}$-a.e. $x \in \Omega$, we have
    \begin{equation*}
        \lim_{r \to 0} r^{1-N} \int_{B(x,r)} \abs{e(u)}^2 \dd{x} = 0.
    \end{equation*}
    This proves step 1.

    \vspace{0.5cm}\noindent
    \emph{Step 2. We show that $$\text{the set $\set{x \in \Sigma(K) | \lim_{r \to 0} \omega(x,r) = 0}$ has a dimension $\leq N-2$.}$$}
    Since $$\Sigma(K) = \set{x \in K | \liminf_{r \to 0} \beta_K(x,r) + \omega(x,r) > 0},$$ we see that
    \begin{equation*}
        \set{x \in \Sigma(K) | \lim_{r \to 0} \omega(x,r) = 0} = \set{x \in K | \liminf_{r \to 0} \beta_K(x,r) > 0, \ \lim_{r \to 0} \omega(x,r) = 0}
    \end{equation*}
    and it can be decomposed as a countable union of sets of the form
    \begin{equation*}
        \Sigma' = \set{x \in K | B(0,R_0) \subset \Omega,\ \forall r \in (0,R_0),\ \beta_K(x,r) > \varepsilon_0 \ \text{and} \ \lim_{r \to 0} \omega(x,r) = 0},
    \end{equation*}
    where $R_0 > 0$ and $\varepsilon_0 > 0$. So let us show that such a set $\Sigma'$ has a dimension $\leq N - 2$.
    We let $s \in (N-2,N-1)$ and we proceed by contradiction by assuming that $\HH^s(\Sigma') > 0$.
    By \cite[Lemma 4.6]{Mattila}, we have $\HH^s_{\infty}(\Sigma') = 0$ and by \cite[Theorem 3.6 (2)]{Simon}, we have
    \begin{equation}\label{eq_sigma_density}
        \limsup_{r \to 0} r^{-s} \HH^s_{\infty}(\Sigma' \cap B(x,r)) > C^{-1} \quad \text{for $\HH^s$-a.e. $x \in \Sigma'$,}
    \end{equation}
    for some constant $C \geq 1$ which depends only on $N$.

    Let us now fix a point $x_0 \in \Sigma'$ such that (\ref{eq_sigma_density}) holds and let $(r_i)_i \to 0$ be a sequence such $(r_i)_i \to 0$ and that for all $i$,
    \begin{equation*}
        r_i^{-s} \HH^s_{\infty}(\Sigma' \cap B(x_0,r_i)) \geq C^{-1}.
    \end{equation*}
    We consider the blow-up sequence $(u_i,K_i)_i$ given by
    \begin{equation*}
        u_i(x) = r_i^{-1/2} u_i(x_0 + r_i x) \quad \text{and} \quad K_i = r_i^{-1} (K - x_0).
    \end{equation*}
    Since $\lim_{r \to 0} h(r) = 0$, we can extract a subsequence (not relabelled) such that $(u_i,K_i)_i$ converges to a global minimizer $(u_{\infty},K_{\infty})$ in $\R^N$, see Section \ref{section_blowup}. Moreover,
    \begin{equation*}
        \lim_{r \to 0} r^{1-N} \int_{B(x,r)} \abs{e(u)}^2 \dd{x} = 0
    \end{equation*}
    so $K_{\infty}$ is a minimal set in $\R^N$.
    Now, we introduce $\Sigma(K_{\infty})$, the singular part of $K_{\infty}$, i.e., the set of points $x \in K_{\infty}$ such that $\liminf_{r \to 0} \beta_{K_{\infty}}(x,r) > 0$.
    By Allard epsilon-regularity theorem, there exists a universal $\varepsilon_1$ such that for all $x \in K_{\infty}$ and $r > 0$, if $\beta_{K_{\infty}}(x,r) \leq \varepsilon_1$, then $K_{\infty}$ is a $C^1$ surface in the neighborhood of $x$.
    This shows that at all points $x \in K_{\infty} \setminus \Sigma(K_{\infty})$, the set $K$ is $C^1$ in a neighborhood of $x$ and thus at such a point, $\lim_{r \to 0} \beta_{K_{\infty}}(x,r) = 0$.
    We also note that according to the regularity theory of minimal sets \cite[Theorem 4.3]{AFH}, we have
    \begin{equation*}
        \mathrm{dim}(\Sigma(K_{\infty})) \leq N - 2
    \end{equation*}
    and thus, since $s > N-2$,
    \begin{equation}\label{eq_minimal_dimension}
        \HH^s(\Sigma(K_{\infty}) = 0.
    \end{equation}
    Next, for all $i$, we set $\Sigma'_i := r_i^{-1} \left(\Sigma' - x_0\right) \subset K_i$. As the flatness is invariant under rescaling, let us note that from the definition of $\Sigma'$ we have
    \begin{equation}\label{eq_sigmai_flatness}
        \text{for all $x \in \Sigma'_i$, for all $r \in (0,r_i^{-1} R_0)$, we have $\beta_{K_i}(x,r) \geq \varepsilon_0$.}
    \end{equation}
    We then check that $\Sigma'_i$ converges to $\Sigma(K)$ in the sense that for all open set $V \subset \R^N$ containing $\Sigma(K_{\infty}) \cap \overline{B}(0,1)$, we have
    \begin{equation}\label{eq_sigma_limit}
        \Sigma'_i \cap \overline{B}(0,1) \subset V \quad \text{for $i$ big enough}.
    \end{equation}
    If (\ref{eq_sigma_limit}) does not hold true, we can find a sequence of points $x_i \in \Sigma'_i \cap \overline{B}(0,1)$ such that for all $i$, $x_i \notin V$. By extracting a subsequence again, we can assume that $(x_i)_i$ converges to some point $x \in \overline{B}(0,1) \setminus V$, which also necessarily belongs to $K_{\infty}$ by convergence of $(K_i)_i$ to $K_{\infty}$.
    Since $x \in K_\infty \cap \overline{B}(0,1) \setminus V \subset K_{\infty} \setminus \Sigma(K_{\infty})$ is a regular point of $K_{\infty}$, there exists $\rho > 0$ such that $\beta_{K_{\infty}}(x,2\rho) < \varepsilon_0/8$.
    By convergence of $(K_i)_i$ to $K$ and $(x_i)_i$ to $x$, one can deduce that for $i$ big enough,
    \begin{equation*}
        \beta_{K_i}(x_i,\rho) \leq \varepsilon_0/2,
    \end{equation*}
    which contradicts (\ref{eq_sigmai_flatness}). This proves (\ref{eq_sigma_limit}).

    Using the fact that
    \begin{equation*}
        C^{-1} \leq r_i^{-s} \HH^s_{\infty}(\Sigma' \cap B(x_0,r_i)) = \HH^s_{\infty}(\Sigma'_i \cap B(0,1))
    \end{equation*}
    and (\ref{eq_sigma_limit}), we see that for all open set $V$ containing $\Sigma(K_\infty) \cap \overline{B}(0,1)$, we have $\HH^s_{\infty}(V) \geq C^{-1}$.
    From the definition of $\HH^s_{\infty}$, one can deduce that $$\HH^s_{\infty}(\Sigma(K_{\infty}) \cap \overline{B}(0,1)) \geq C^{-1}.$$ We finally arrive at
    \begin{equation*}
        \HH^s(\Sigma(K_{\infty}) \cap \overline{B}(0,1)) \geq \HH^s_{\infty}(\Sigma(K_{\infty}) \cap \overline{B}(0,1)) \geq C^{-1},
    \end{equation*}
    which contradicts (\ref{eq_minimal_dimension}).
    We conclude that for all $s \in (N-2,N-1)$, we have $\HH^s(\Sigma') = 0$ and thus $\mathrm{dim}(\Sigma') \leq N-2$.
\end{proof}

\begin{appendices}

    \section{Auxiliary lemmas about affine maps}

    This section is devoted to justifying a few elementary properties of affine maps.
    Our first result controls the $L^{\infty}$ norm of an affine map on $B(0,R)$ by its average value on a subset $E \subset B(0,R)$. Similar and more general estimates of this kind are also proved in \cite[Lemma 3.4]{FS2}, \cite{Friedrich2}, \cite{FS}.

    \begin{lemma}\label{lem_int_rigid}
        For all real number $p \geq 1$, for all constant $c \in \R$ and vector $v \in \R^N$, for all radius $R > 0$ and for all Borel set $E \subset B(0,R) \subset \R^N$, we have
        \begin{equation*}
            \fint_E \abs{c + v \cdot x}^p \dd{x} \geq C^{-1} \left(\abs{c}^p + R^p \abs{v}^p\right) \left(\frac{\abs{E}}{R^N}\right)^p,
        \end{equation*}
        where $C \geq 1$ depends on $N$ and $p$.
    \end{lemma}
    \begin{proof}
        In view of the homogeneity of the inequality, we can assume $R = 1$ without loss of generality.
        We start by proving a simpler inequality, namely, that there exists a constant $C \geq 1$ (depending on $N$ and $p$) such that
        \begin{equation}\label{eq_E0}
            \int_E \abs{c + v \cdot x}^p \dd{x} \geq C^{-1} \abs{v}^p \abs{E}^{p+1}.
        \end{equation}
        Without loss of generality, we can assume $v \ne 0$.
        For $\delta > 0$, the inequality $\abs{c + v \cdot x} \leq \delta$ defines a $\delta \abs{v}^{-1}$-neighborhood of some affine hyperplane so there exists a constant $C > 0$ (depending on $N$) such that for all $\delta > 0$
        \begin{equation*}
            \abs{B(0,1) \cap \set{\abs{c + v \cdot x} \leq \delta}} \leq C \abs{v}^{-1} \delta.
        \end{equation*}
        Therefore we can estimate for $\delta > 0$,
        \begin{align*}
            \abs{E} &\leq \abs{E \cap \set{\abs{c + v \cdot x} \leq \delta}} + \abs{E \cap \set{\abs{c + v \cdot x} \geq \delta}}\\
                    &\leq C \abs{v}^{-1} \delta + \abs{E \cap \set{\abs{c + v \cdot x} \geq \delta}}.
        \end{align*}
        We choose $\delta := (2C)^{-1} \abs{v} \abs{E}$ so that
        \begin{equation*}
            \abs{E \cap \set{\abs{c + v \cdot x} \geq \delta}} \geq \frac{1}{2} \abs{E}.
        \end{equation*}
        Then we have
        \begin{align*}
            \int_E \abs{c + v \cdot x}^p \dd{x}  &\geq \int_{E \cap \set{\abs{c + v \cdot x} \geq \delta}} \abs{c + v \cdot x}^p \dd{x}\\
                                                 &\geq 2^{-1} \delta^p \abs{E}\\
                                                 &\geq 2^{-p-1} C^{-p} \abs{v}^p \abs{E}^{p+1},
        \end{align*}
        which proves our claim.
        Now, we pass to the proof of the general inequality. If $c \geq 2 \abs{v}$, then for all $x \in E \subset B(0,1)$, we have $\abs{c + v \cdot x} \geq c/2$ so
        \begin{equation*}
            \int_E \abs{c + v \cdot x}^p \dd{x} \geq 2^{-p} c^p \abs{E} \geq 2^{-p} \left(\frac{c^p + \abs{v}^p}{1 + 2^{-p}}\right) \abs{E}.
        \end{equation*}
        Note that we can also bound from below $\abs{E} \geq \abs{B(0,1)}^{-p} \abs{E}^{p+1}$ since $E \subset B(0,1)$.
        If $c \leq 2 \abs{v}$, we use (\ref{eq_E0}), which gives
        \begin{equation*}
            \int_E \abs{c + v \cdot x}^p \dd{x} \geq C^{-1} \abs{v}^p \abs{E}^{p+1} \geq C^{-1} \left(\frac{\abs{c}^p + \abs{v}^p}{1 + 2^p}\right) \abs{E}^{p+1}.
        \end{equation*}
    \end{proof}

    We shall need an analogue inequality on the unit sphere.
    \begin{lemma}\label{lem_int_rigid2}
        For all real number $p \geq 1$, for all Borel set $E \subset \partial B(0,1)$, for all vector $b \in \R^N$ and matrix $A \in \R^{N \times N}$, we have
        \begin{equation*}
            \fint_{E} \abs{b \cdot x}^p + \abs{A x}^p \dd{\HH^{N-1}(x)} \geq C^{-1} \left(\abs{b}^p + |A|^p\right) \HH^{N-1}(E)^p,
        \end{equation*}
        for some constant $C \geq 1$ which depends on $N$ and $p$.
    \end{lemma}
    \begin{proof}
        In the case $A = 0$, the proof is exactly like Lemma \ref{lem_int_rigid} with $B(0,1)$ replaced by $\partial B(0,1)$.
        We then pass to the case $b = 0$.
        For $i=1,\ldots,N$, we let $A_i$ denote the $i$-th column of $A^T$. From the definition of the Frobenius norm, we see that $\abs{A} \leq C \max_i \abs{A_i}$.
        We can thus fix an index $k$ such that $\abs{A_k} \geq C^{-1} \abs{A}$.
        We observe that for $x \in \R^N$
        \begin{equation*}
            \abs{Ax} = \sqrt{\sum_i (A_i \cdot x)^2} \geq \abs{A_k \cdot x} = \abs{A_k} \ \abs{b \cdot x},
        \end{equation*}
        where $b \in \R^N$ is a unit vector, so an application of the first step concludes that
        \begin{equation*}
            \fint_E \abs{A x}^p \dd{\HH^{N-1}} \geq C^{-1} |A|^p \HH^{N-1}(E)^p.
        \end{equation*}
    \end{proof}

\end{appendices}

\section*{Acknowledgements}

Camille Labourie was funded by the French National Research Agency (ANR) under grant ANR-21-CE40-0013-01 (project GeMfaceT).

\bibliographystyle{plain}
\bibliography{biblio_griffith}

\end{document}